\documentclass[a4paper,12pt,reqno,oneside]{amsart}
\usepackage[T1]{fontenc}
\usepackage[utf8]{inputenc}
\usepackage{typearea}
\usepackage{amsmath,amssymb,amsthm}
\usepackage[dvipsnames]{xcolor}
\usepackage{hyperref}
\usepackage[capitalise,nameinlink]{cleveref}
\usepackage{mathtools}
\usepackage{indentfirst}
\usepackage{enumitem}
\usepackage{xfrac}
\usepackage{caption}
\usepackage{subcaption}
\usepackage[section]{placeins}
\usepackage{tikz}
\makeindex

\graphicspath{ {./img/} }

\crefformat{enumi}{#2(#1)#3}

\definecolor{color0}{RGB}{34,113,178}
\definecolor{color1}{RGB}{61,183,233}
\definecolor{color2}{RGB}{247,72,165}
\definecolor{color3}{RGB}{53,155,115}
\definecolor{color4}{RGB}{213,94,0}
\definecolor{color5}{RGB}{230,159,0}
\definecolor{color6}{RGB}{114,228,97}
\definecolor{color7}{RGB}{53,37,94}
\definecolor{color8}{RGB}{128,0,103}

\colorlet{blue}{color0}
\colorlet{pink}{color2}

\hypersetup{
  pagebackref,
  colorlinks,
  linkcolor=blue,
  citecolor=pink,
  urlcolor=pink
}

\theoremstyle{plain}
\newtheorem{theorem}{Theorem}
\newtheorem*{theorem*}{``Theorem''}

\newtheorem{lemma}{Lemma}
\newtheorem*{lemma*}{Lemma}

\theoremstyle{definition}
\newtheorem*{definition}{Definition}
\newtheorem{example}{Example}
\newtheorem*{remark}{Remark}

\theoremstyle{plain}
\makeatletter
\newtheorem*{re@theorem}{\re@title}
\newenvironment{retheorem}[1]{%
 \def\re@title{\cref{#1}}%
 \begin{re@theorem}}%
 {\end{re@theorem}}
\makeatother

\newcommand{\from}{\colon}
\newcommand{\xto}[1]{\xrightarrow{#1}}
\newcommand{\wt}[1]{\widetilde{#1}}
\newcommand{\Conf}{\mathrm{Conf}}
\newcommand{\id}{\mathrm{id}}

\newcommand{\R}{\mathbb{R}}
\newcommand{\RP}{\mathbb{R}P}
\newcommand{\Z}{\mathbb{Z}}
\newcommand{\N}{\mathbb{N}}
\newcommand{\CC}{\mathcal{C}}
\newcommand{\TT}{\mathbb{T}}
\newcommand{\Eq}{\mathrm{Eq}}
\newcommand{\HH}{\mathcal{H}}
\newcommand{\JJ}{\mathcal{J}}

\newcommand{\gbox}{\tikz{\node[shape=rectangle,fill=black!20,draw,inner sep=0.7mm,align=center] (char) {};}}
\newcommand{\wbox}{\tikz{\node[shape=rectangle,draw,inner sep=0.7mm,align=center] (char) {};}}
\newcommand*\Gc{G^{\hspace{1pt}\raisebox{0.8pt}{\gbox}}}
\newcommand*\wGc{\wt{G}^{\hspace{1pt}\raisebox{0.8pt}{\gbox}}}
\newcommand*\Gcf{G^{\hspace{1pt}\raisebox{0.8pt}{\wbox}}_{f}}
\newcommand*\wGcf{\wt{G}^{\hspace{1pt}\raisebox{0.8pt}{\wbox}}_{f}}

\DeclareEmphSequence{\bfseries,\itshape,\upshape}

\DeclareMathOperator{\pr}{pr}
\DeclareMathOperator{\sk}{sk}
\DeclareMathOperator{\st}{st}
\DeclareMathOperator{\ord}{ord}
\DeclareMathOperator{\Int}{Int}
\DeclareMathOperator{\codim}{codim}

\author{Alexey Gorelov}
\title{Lifting maps between graphs to embeddings}
\address{Université Grenoble Alpes, CNRS, Institut Fourier, 38000 Grenoble, France}
\email{algor512@gmail.com}

\begin{document}

\maketitle

\begin{abstract}
  In this paper, we study conditions for the existence of an embedding
  $\wt{f} \from P \to Q \times \R$ such that $f = \pr_Q \circ \wt{f}$, where $f \from P \to Q$ is a
  piecewise linear map between polyhedra. Our focus is on non-degenerate maps between graphs, where
  non-degeneracy means that the preimages of points are finite sets.

  We introduce combinatorial techniques and establish necessary and sufficient conditions for the
  general case. Using these results, we demonstrate that the problem of the existence of a lifting reduces
  to testing the satisfiability of a 3-CNF formula. Additionally, we construct
  a counterexample to a result by V.~Po\'{e}naru on lifting of smooth immersions to
  embeddings.

  Furthermore, by establishing connections between the stated problem and the approximability by
  embeddings, we deduce that, in the case of generic maps from a tree to a segment, a weaker
  condition becomes sufficient for the existence of a lifting.
\end{abstract}

\section{Introduction}

\subsection{Background}

In this paper, we consider the following problem: given compact one-dimensional polyhedra $P$ and
$Q$, and a piecewise linear map $f \from P \to Q$, what conditions must be satisfied for the
existence of a piecewise linear embedding $\wt{f} \from P \to Q \times \R$ such that
$f = \pr_Q \circ \wt{f}$? If such an embedding exists, we refer to it as a \emph{(codimension one)
  lifting of $f$ to an embedding}, and we say that $f$ \emph{lifts to an embedding}\index{lifting}. This problem
can also be formulated in other contexts, such as for smooth maps between manifolds or continuous
maps between topological spaces. Additionally, it is possible to consider liftings of codimension
$k$, namely embeddings $\widetilde{f} \from P \to Q \times \mathbb{R}^k$ such that
$f = \pr_Q \circ \widetilde{f}$. In this case, maps that admit such lifting are sometimes referred
to as $k$-prems\index{$n$-prem} (short for $k$-projected embedding), see~\cite{szucs,akhmet,mel}.

In this paper, our focus is on non-degenerate maps. A piecewise linear map $f \from P \to Q$ is
called \emph{non-degenerate}\index{non-degeneracy} if the set $f^{-1}(q)$ is finite for every point $q \in
Q$. Additionally, we call a simplicial map $f \from K \to L$ \emph{non-degenerate} if the
corresponding piecewise linear map $|f| \from |K| \to |L|$ is non-degenerate. It is easy to see that
a simplicial map $f \from K \to L$ between finite simplicial complexes is non-degenerate if and only
if it is injective on each simplex; that is, a $k$-simplex $A \in K$ maps to a $k$-simplex
$f(A) \in L$.

In the one-dimensional case, non-degenerate simplicial maps are synonymous with graph
homomorphisms. This establishes an interesting connection between the problem we are addressing in
this paper and certain problems in graph theory. In fact, a graph homomorphism $f \from G \to H$ is
known as an $H$-colouring of the graph $G$. Each lifting of the map $f$ to an embedding introduces
orders on sets of vertices of the same colour, which, in turn, determine the orders on the
edges. These constructions, which involve colourings along with orders, are actively studied in
graph theory, see e.g.~\cite{track}. For example, so-called track layouts\index{track layout} can be viewed as liftings
of graph homomorphisms to complete graphs.\footnote{It is worth noting that the author has not come
  across research on track layouts for other types of colouring.}

It should be highlighted that the one-dimensional piecewise linear version of the lifting to an
embedding problem plays a special role among other versions of this problem. For instance, the
problem of liftability to an embedding of smooth immersions between manifolds can be reduced to the
one-dimensional piecewise linear version of the problem~\cite{carter_saito_book,poen_fr}. Also, as
we will show below in~\cref{thm:lifting_sk}, the multidimensional piecewise linear version of the
problem reduces to the one-dimensional case.

\subsection{Structure}

Let us now outline the general structure of the paper.

The second section introduces several combinatorial concepts that are utilized throughout the rest
of the paper. In particular, we prove the following necessary conditions for the existence of a
lifting to an embedding:

\begin{retheorem}{thm:obstructors}
  Let $f \from K \to L$ be a non-degenerate simplicial map between finite simplicial complexes such
  that $|f| \from |K| \to |L|$ lifts to an embedding. Then the following holds:
  \begin{enumerate}
  \item all the covering maps $p_n \from |K^{(n)}_f| \to |\wt{K}^{(n)}_f|,\ n > 1$ are trivial;
  \item there are no $n$-obstructors for $f$ for any $n > 1$.
  \end{enumerate}
\end{retheorem}

Here $K^{(n)}_f$ and $\wt{K}^{(n)}_f$ denote simplicial models for the ordered and unordered
configuration spaces of $n$ points with the same image under $f$, and an $n$-obstructor is a path in
$K^{(n)}_f$ that realises a cyclic permutation of points; the reader may refer to the next section
for the formal definitions.

Further, we prove two necessary and sufficient conditions for the existence of a lifting.

\begin{retheorem}{thm:orders}
  Let $f \from K \to L$ be a non-degenerate simplicial map between finite simplicial complexes. Then
  the piecewise linear map $|f| \from |K| \to |L|$ lifts to an embedding if and only if there is a
  collection of mutually compatible linear orders on the sets $f^{-1}(v),\ v \in V(L)$.

  Furthermore, there is a bijection between such collections of linear orders and the isotopy classes of liftings.
\end{retheorem}

Informally, the compatibility here means that the orders on the vertices induce orders on the
simplices; again, we postpone the formal definition until the second section.

\begin{retheorem}{thm:gamma_cond}
  Let $f \from K \to L$ be a non-degenerate simplicial map between finite simplicial
  complexes. Assuming that the covering map $p_2 \from |K^{(2)}_f| \to |\wt{K}^{(2)}_f|$ is trivial,
  there is a 3-CNF formula $\Gamma_f$,\footnote{A boolean formula $\Gamma$ is said to be in 3-CNF\index{3-CNF} if
    it can be expressed as $\Gamma = \bigwedge_{i} (\alpha_{i} \vee \beta_{i} \vee \gamma_{i})$,
    where each of $\alpha_{i}$, $\beta_{i}$, and $\gamma_{i}$ is either a variable or its negation.}
  such that the piecewise linear map $|f| \from |K| \to |L|$ lifts to an embedding if and only if
  $\Gamma_f$ is satisfiable. Furthermore, there is a bijection between the assignments that satisfy
  $\Gamma_f$ and the isotopy classes of liftings.
\end{retheorem}

We will explicitly describe the formula $\Gamma_f$ just before we state~\cref{thm:gamma_cond} in the
second section.

In the third section, we explore connections between the problem of lifting of maps between graphs
and the problem of lifting of smooth immersions to embeddings. In particular we show that the
assumptions in a result stated by V.~Po\'{e}naru in~\cite{poen_fr} need to be strengthened by
including the condition of the satisfiability of $\Gamma_{f}$. The main result in this section
demonstrates that without this stronger assumption, the main result on the existence of liftings of
immersions stated in~\cite{poen_fr} cannot hold:

\begin{retheorem}{thm:realisation_gamma}
  Each 3-CNF formula $\Gamma$ of a specific form, as described in the statement
  of~\cref{thm:realisation_gamma} on page~\pageref{thm:realisation_gamma}, can be realised
  as $\Gamma_f$ of a map $f \from G \to H$ between graphs. Moreover, there exists a generic
  immersion $g \from S \looparrowright B$ of a surface $S$ with boundary into a handlebody $B$, such
  that $H$ is the graph of multiple points of $g$, and $f$ is the restriction of $g$ on $f^{-1}(H)$.
\end{retheorem}

In the last section, we establish a connection between the problem of the existence of
liftings of maps between graphs and the problem of approximation by embeddings. By using this
connection, we deduce the main result of the section:

\begin{retheorem}{thm:complete_2obs}
  For the generic simplicial maps from a tree to a segment the non-existence of 2-obstructors is a
  necessary and sufficient condition for the existence of a lifting.
\end{retheorem}

\subsection{Conventions}

Throughout this paper, we adhere to the following notation and conventions.

We use the letters $K$ and $L$ to represent finite simplicial complexes, with $|K|$ and $|L|$
denoting their geometric realisations. All maps between simplicial complexes are assumed to be
simplicial. Additionally, for a simplicial map $f \from K \to L$, we denote the induced piecewise
linear map between $|K|$ and $|L|$ by $|f|$.

Given a simplicial complex or a graph $K$, we use $V(K)$ and $E(K)$ to denote the set of vertices and
the set of edges of $K$, respectively. Moreover, $\sk_n K$ denotes the $n$-skeleton of $K$ --- a
simplicial complex consisting of all simplices of $K$ of dimension lower than or equal to $n$.

As previously noted, non-degenerate simplicial maps between one-dimensional simplicial complexes are
essentially graph homomorphisms. Occasionally, particularly in examples, it will be more convenient
to consider multigraph homomorphisms.

Let $G$ and $H$ be multigraphs.\index{multigraph}\footnote{Graphs that can contain loops and
  multiple edges.}  A pair $(f_V, f_E)$ of maps $f_V \from V(G) \to V(H)$ and
$f_E \from E(G) \to E(H)$ is called a \emph{multigraph homomorphism}\index{multigraph!homomorphism}
$f \from G \to H$ if, for any edge $e \in E(G)$, the map $f_V$ takes the endpoints of $e$ to the
endpoints of $f_E(e)$.

As a graph, a multigraph $G$ has a geometric realisation denoted by $|G|$. Specifically, $|G|$ is the
union of $|V(G)|$ points $p_{v}$ representing the vertices $v \in V(G)$, and $|E(G)|$ segments
$I_{e}$ representing the edges $e \in E(G)$, where the segments are glued to the points according to
the incidence relation.\footnote{Certainly, loops are represented by loops (or segments with
  both endpoints identical) in $|G|$.} Clearly, $|G|$ is a one-dimensional polyhedron uniquely
defined up to piecewise linear homomorphism.

Furthermore, a multigraph homomorphism $f \from G \to H$ induces a piecewise linear map $|f|$
between the geometric realisations of $G$ and $H$, which maps $p_{v} \in |G|$ to
$p_{f_{V}(v)} \in |H|$ and homeomorphically maps $I_{e}$ to $I_{f_{E}(e)}$. Thus, by taking
triangulations of $|G|$ and $|H|$ in which $|f|$ is simplicial, we can obtain a simplicial
map. Therefore, whenever multigraph homomorphisms are used in this text, they can be replaced by
simplicial maps after subdividing the involved multigraphs. However, to simplify the notation, we
choose to use multigraph homomorphisms in some places of the text.

\section{Combinatorics of liftings to embeddings}

In this section, we will discuss combinatorial techniques that are useful for studying the existence
of liftings to embeddings of maps between graphs.

Following our conventions, let $K$ and $L$ be simplicial complexes, and $f \from K \to L$ be a
non-degenerate simplicial map between them. Let us denote by $K^{(n)}_f$ a simplicial complex whose
vertices are $n$-tuples of distinct vertices of $K$ that map to the same vertex by $f$.
Additionally, a set $\{(v_1^1, \dots, v_n^1), \dots, (v_1^{k+1}, \dots, v_n^{k+1})\}$ of vertices
forms a $k$-simplex if for each $j = 1, \dots, n$ the set $\{v_j^1, \dots, v_j^{k+1}\}$ forms a
$k$-simplex $A_j$ in $K$. Note that $A_j \cap A_s = \varnothing$ for all $j \neq s$, and all the
simplices $A_j$ have the same image under $f$.

The symmetric group $S_n$ acts naturally on $K^{(n)}_f$ by permuting the points in $n$-tuples. Thus,
we can define the unordered version of $K^{(n)}_f$, which we denote by
$\wt{K}^{(n)}_f = K^{(n)}_f / S_n$. Clearly, $S_n$ also acts on the geometric realisation
$|K^{(n)}_f|$ of $K^{(n)}_f$. This action is properly discontinuous, so it induces a covering map
$p_n \from |K^{(n)}_f| \to |\wt{K}^{(n)}_f|$ that forgets the order of points. Moreover, it is
clear that $p_n$ is a principal $S_n$-bundle.\footnote{For a short introduction to this topic the
  reader may refer to~\cite{principal_bundles}.}

\begin{definition}[$n$-obstructor]
  An \emph{$n$-obstructor}\index{$n$-obstructor} for $f \from |K| \to |L|$ is a vertex
  $(x_1, x_2, \dots, x_n) \in K^{(n)}_f$ and a path from $(x_1, x_2, \dots, x_n)$ to
  $(x_n, x_1, x_2, \dots, x_{n-1})$ in $K^{(n)}_f$.
\end{definition}

The reader may refer to~\cref{ex:siek} below to see an example of 2-obstructor.

\begin{lemma}\label{lemma:equiv_obstr}
  For any non-degenerate simplicial map $f \from K \to L$ and $N > 1$, the following statements are equivalent:
  \begin{enumerate}
  \item all the covering maps $p_k \from |K^{(k)}_f| \to |\wt{K}^{(k)}_f|$, $1 < k \leq N$ are trivial;
  \item there are no $k$-obstructors for $f$ for $1 < k \leq N$.
  \end{enumerate}
\end{lemma}

\begin{proof}
  The implication $(1) \Rightarrow (2)$ is straightforward. Indeed, assume there is a
  $k$-obstructor that defines a path $\gamma \from I \to |K^{(k)}_f|$. Then $p_k \circ \gamma$ is a
  loop in $|\wt{K}^{(k)}_f|$, but $\gamma$ is not a loop in $|K^{(k)}_f|$, implying that $p_k$ is not trivial.

  Now, let us prove the implication $(2) \Rightarrow (1)$. Assume there are no $k$-obstructors for
  any $1 < k \leq N$. Let $|K^{(k)}_f| = \bigsqcup_i C_i$, where $\{ C_i \}_i$ are connected
  components of $|K^{(k)}_f|$.\footnote{Since the polyhedra are locally path-connected spaces,
    connectedness and path-connectedness are equivalent in this case.}

  Clearly, the action of $S_k$ on $|K^{(k)}_f|$ induces an action of $S_k$ on $\{ C_i \}_i$. Let us
  prove that this action is free. Suppose it is not. Then there exists $\sigma \neq \id \in S_k$ and
  a connected component $C$ such that $\sigma(C) = C$. Let $p$ be a prime divisor of $\ord(\sigma)$,
  and define $\sigma' = \sigma^{\sfrac{\ord(\sigma)}{p}}$. Clearly,
  $\sigma'(C) = \sigma^{\sfrac{\ord(\sigma)}{p}}(C) = C$. Moreover, since $\ord(\sigma') = p$, the
  cycle decomposition of $\sigma'$ consists of cycles of length one, and one or more disjoint cycles
  of length $p$. Let $(s_1\; s_2\; \dots\; s_p)$ be one of them.

  Now consider any vertex $x \in C$. As $\sigma'(x) \in C$, there exists a path $\gamma$ in $C$
  connecting $x$ and $\sigma'(x)$. In fact, it defines $k$ paths in $K$
  \[
    \begin{aligned}
      \gamma_1 \from x_1 &\to y_1^{1} &\to \dots &\to y_1^{l} &\to x_{\sigma'(1)} \\
      \gamma_2 \from x_2 &\to y_2^{1} &\to \dots &\to y_2^{l} &\to x_{\sigma'(2)} \\
      \vdots             &            & \vdots   &            &\vdots\hspace{1em} \\
      \gamma_k \from x_k &\to y_k^{1} &\to \dots &\to y_k^{l} &\to x_{\sigma'(k)} \\
    \end{aligned}
  \]
  such that $y_i^k \neq y_j^k$ for $i \neq j$, and $f(y_i^k) = f(y_j^k)$. By taking the paths $\gamma_{s_1}, \dots,
  \gamma_{s_p}$, we obtain a path in $K^{(p)}_f$. It can be seen that this path defines a
  $p$-obstructor.

  Therefore, the action of $S_n$ on $\{ C_i \}_i$ is free. Let $\{ O_\alpha \}_\alpha$ be
  the set of orbits of the action. For each orbit $O_\alpha$, fix an arbitrary connected component
  $C_\alpha \in O_\alpha$, and let $T = \bigsqcup_{\alpha} C_\alpha$.

  For every point $x \in |\wt{K}^{(k)}_f|$, the preimages $p_k^{-1}(x)$ form an orbit of the action of
  $S_{n}$ on $|K^{(k)}_f|$. Thus, the connected components containing these preimages also form an
  orbit. Moreover, since the action of $S_{n}$ on them is free, two distinct preimages cannot lie in the
  same connected component. The construction of $T$ ensures that exactly one of the points in
  $p_k^{-1}(x)$ lies in some of $C_\alpha \subset T$. Therefore, the restriction $p_k \big|_T$ is bijective,
  and we can define a map $s \from |\wt{K}^{(k)}_f| \to |K^{(k)}_f|$ by putting
  $s(x) = \left(p_k \big|_T\right)^{-1}(x)$.

  Note that $p_k$ is an open map; indeed, for any open set $U \subset |K^{(k)}_f|$ we have
  $p_k^{-1}(p_k(U)) = \bigcup_{\sigma \in S_k} \sigma(U)$, and all the sets $\sigma(U)$ are
  open. Thus, for an open $U \subseteq |K^{(k)}_f|$, we have $s^{-1}(U) = s^{-1}(U \cap T) = p_k(U \cap
  T)$ which is open since $p_k$ is open, and both $U$ and $T$ are open.

  Therefore, $s$ is a continuous map defining a section of a principal $S_k$-bundle $p_k$, which
  implies that $p_k$ is trivial by~\cite[Proposition 2.1]{principal_bundles}.
\end{proof}

Now, we will state necessary conditions for the existence of a lifting. It is worth mentioning that
the necessity of~\cref{thm:obstructors:second}, for the case of generic immersions between manifolds
and in a slightly different formulation, is proven in~\cite[Lemme 1.1]{poen_fr}. The necessity
of~\cref{thm:obstructors:first} for $n=2$ was discussed in~\cite[\S 3.1]{akh_rep_skop} for maps to
$\R^N$ and proven in~\cite[Corollary 3]{giller} for generic immersions of surfaces into $\R^3$.

\begin{theorem}\label{thm:obstructors}
  Let $f \from K \to L$ be a non-degenerate simplicial map such that $|f| \from |K| \to |L|$ lifts
  to an embedding. Then the following holds:
  \begin{enumerate}
  \item\label{thm:obstructors:first} all the covering maps $p_n \from |K^{(n)}_f| \to |\wt{K}^{(n)}_f|,\ n > 1$ are trivial;
  \item\label{thm:obstructors:second} there are no $n$-obstructors for $f$ for any $n > 1$.
  \end{enumerate}
\end{theorem}

\begin{proof}
  Let us begin by noting that the points of $|K^{(n)}_f|$ and $|\wt{K}^{(n)}_f|$ can be
  considered as ordered and unordered $n$-tuples of distinct points of $|K|$ with the same image.

  Suppose there is a lifting $\wt{|f|} \from |K| \to |L| \times \R$. Let
  $h = \pr_R \wt{|f|} \from |K| \to \R$. This induces a continuous map that takes $(x_1, \dots, x_n)$
  to $(h(x_1), \dots, h(x_n))$. Since $\wt{|f|}$ is an embedding, all $h(x_i)$'s are distinct. Thus,
  it actually maps $|K^{(n)}_f|$ to the configuration space $\Conf_n(\R)$ of $n$ distinct points of
  $\R$. Moreover, it is equivariant with respect to the actions of $S_n$ on both
  spaces.\footnote{$F \from X \to Y$ is equivariant with respect to actions of a group $G$ on $X$
    and $Y$ if $F(g(x)) = g(F(x))$ for any $x \in X$ and $g \in G$.} It is easy to see that $\Conf_n(\R)$
  equivariantly deformation retracts to the discrete space $\Conf_{n}(\{1, 2, \dots, n\})$ of all $n!$ permutations of the
  numbers $\{1, \dots, n\}$ with the natural action of $S_n$. By taking an equivariant isomorphism
  that sends the $n$-tuple $(1, 2, \dots, n)$ to $\id \in S_{n}$, we can identify $\Conf_{n}(\{1, 2, \dots, n\})$ with $S_n$.

  Therefore, we have an equivariant map $H \from |K^{(n)}_f| \to S_n$. Now define
  $g \from |K^{(n)}_f| \to S_n \times |\wt{K}^{(n)}_f|$ as $g(x) = (H(x), p_n(x))$. Clearly, $g$ is
  a morphism between the principal $S_n$-bundles $p_n$ and the trivial bundle
  $\pr_{|\wt{K}^{(n)}_f|} \from S_n \times |\wt{K}^{(n)}_f| \to |\wt{K}^{(n)}_f|$, thus, it is an
  isomorphism by~\cite[Proposition 2.1]{principal_bundles}. Therefore, $p_n$ is trivial.

  The non-existence of $n$-obstructors then follows from~\cref{lemma:equiv_obstr}.
\end{proof}

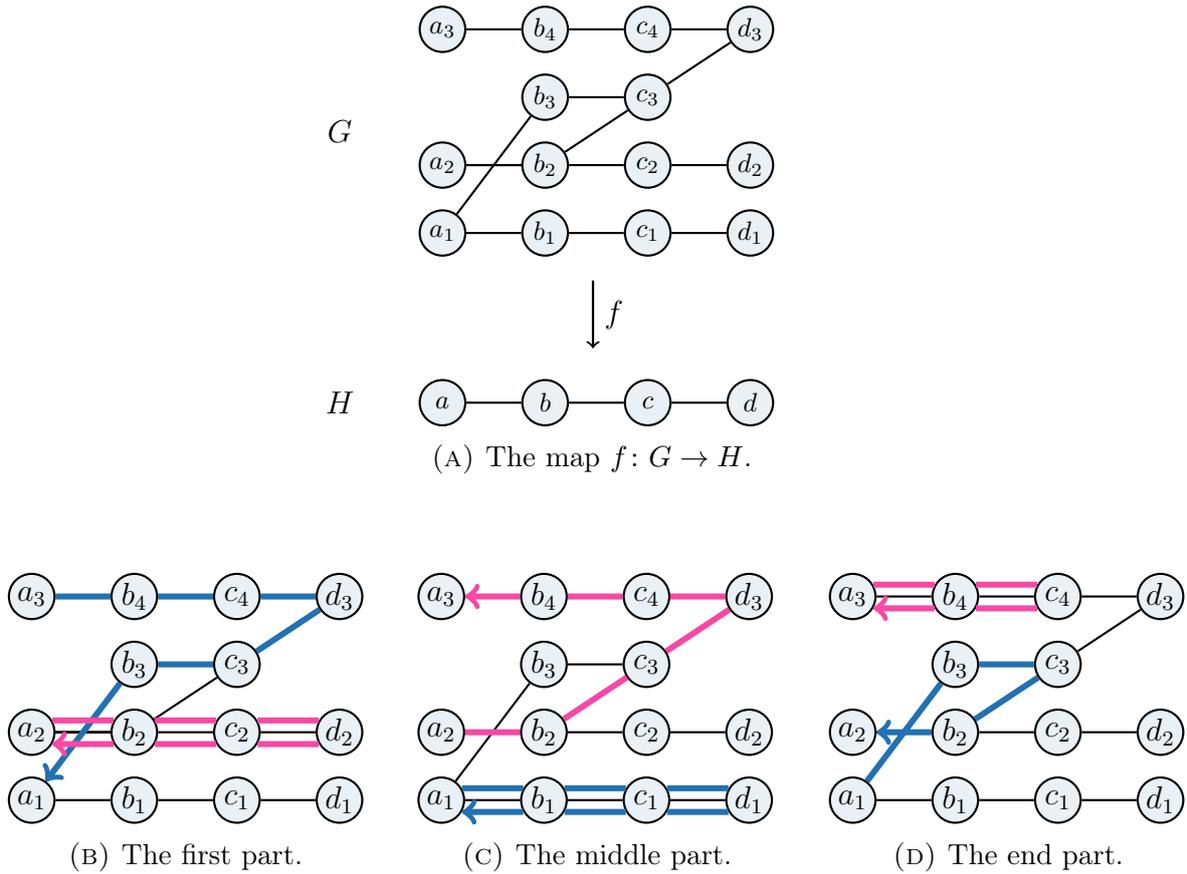
\begin{figure}[htb!]
  \centering
  \begin{subfigure}{0.99\textwidth}
    \centering
  \hspace{-3em}\begin{tikzpicture}[scale=0.9,thick,main
    node/.style={circle,fill=blue!10,draw,font=\sffamily\footnotesize\bfseries,minimum size=0.6cm,inner sep=0pt}]

    \node at (-1.5, 5) (K) {$G$};

    \node[main node] at (0, 6.5) (a3) {$a_3$};
    \node[main node] at (1.5, 6.5) (b4) {$b_4$};
    \node[main node] at (3, 6.5) (c4) {$c_4$};
    \node[main node] at (4.5, 6.5) (d3) {$d_3$};

    \node[main node] at (1.5, 5.5) (b3) {$b_3$};
    \node[main node] at (3, 5.5) (c3) {$c_3$};

    \node[main node] at (0, 4.5) (a2) {$a_2$};
    \node[main node] at (1.5, 4.5) (b2) {$b_2$};
    \node[main node] at (3, 4.5) (c2) {$c_2$};
    \node[main node] at (4.5, 4.5) (d2) {$d_2$};

    \node[main node] at (0, 3.5) (a1) {$a_1$};
    \node[main node] at (1.5, 3.5) (b1) {$b_1$};
    \node[main node] at (3, 3.5) (c1) {$c_1$};
    \node[main node] at (4.5, 3.5) (d1) {$d_1$};

    \node at (-1.5, 1) (L) {$H$};
    \node[main node] at (0, 1) (a) {$a$};
    \node[main node] at (1.5, 1) (b) {$b$};
    \node[main node] at (3, 1) (c) {$c$};
    \node[main node] at (4.5, 1) (d) {$d$};

    \draw (a3) -- (b4) -- (c4) -- (d3) -- (c3) -- (b3) -- (a1) -- (b1) -- (c1) -- (d1);
    \draw (c3) -- (b2) -- (a2) -- (b2) -- (c2) -- (d2);
    \draw (a) -- (b) -- (c) -- (d);

    \draw [->,line width=0.3mm] (2.2,2.8) -- (2.2,1.8) node[pos=0.5,right]{$f$};
  \end{tikzpicture}
  \caption{The map $f \from G \to H$.}
\end{subfigure}
\vspace{3em}

\begin{subfigure}{0.3\textwidth}
  \centering
    \begin{tikzpicture}[scale=0.9,thick,main
      node/.style={circle,fill=blue!10,draw,font=\sffamily\bfseries,minimum
        size=0.6cm,inner sep=0pt}]

    \node[main node] at (0, 6.5) (a3) {$a_3$};
    \node[main node] at (1.5, 6.5) (b4) {$b_4$};
    \node[main node] at (3, 6.5) (c4) {$c_4$};
    \node[main node] at (4.5, 6.5) (d3) {$d_3$};

    \node[main node] at (1.5, 5.5) (b3) {$b_3$};
    \node[main node] at (3, 5.5) (c3) {$c_3$};

    \node[main node] at (0, 4.5) (a2) {$a_2$};
    \node[main node] at (1.5, 4.5) (b2) {$b_2$};
    \node[main node] at (3, 4.5) (c2) {$c_2$};
    \node[main node] at (4.5, 4.5) (d2) {$d_2$};

    \node[main node] at (0, 3.5) (a1) {$a_1$};
    \node[main node] at (1.5, 3.5) (b1) {$b_1$};
    \node[main node] at (3, 3.5) (c1) {$c_1$};
    \node[main node] at (4.5, 3.5) (d1) {$d_1$};

    \draw (a3) -- (b4) -- (c4) -- (d3) -- (c3) -- (b3) -- (a1) -- (b1) -- (c1) -- (d1);
    \draw (c3) -- (b2) -- (a2) -- (b2) -- (c2) -- (d2);
    \draw[->,color=color0,line width=0.8mm] (a3) -- (b4) -- (c4) -- (d3) -- (c3) -- (b3) -- (a1);

    \draw[color=color2,line width=0.8mm] (a2.30) -- (b2.150);
    \draw[color=color2,line width=0.8mm] (b2.30) -- (c2.150);
    \draw[color=color2,line width=0.8mm] (c2.30) -- (d2.150);
    \draw[color=color2,line width=0.8mm] (d2.210) -- (c2.-30);
    \draw[color=color2,line width=0.8mm] (c2.210) -- (b2.-30);
    \draw[->,color=color2,line width=0.8mm] (b2.210) -- (a2.-30);
  \end{tikzpicture}
  \caption{The first part.}
\end{subfigure}
\hspace{1em}
\begin{subfigure}{0.3\textwidth}
  \begin{tikzpicture}[scale=0.9,thick,main
    node/.style={circle,fill=blue!10,draw,font=\sffamily\bfseries,minimum size=0.6cm,inner sep=0pt}]

    \node[main node] at (0, 6.5) (a3) {$a_3$};
    \node[main node] at (1.5, 6.5) (b4) {$b_4$};
    \node[main node] at (3, 6.5) (c4) {$c_4$};
    \node[main node] at (4.5, 6.5) (d3) {$d_3$};

    \node[main node] at (1.5, 5.5) (b3) {$b_3$};
    \node[main node] at (3, 5.5) (c3) {$c_3$};

    \node[main node] at (0, 4.5) (a2) {$a_2$};
    \node[main node] at (1.5, 4.5) (b2) {$b_2$};
    \node[main node] at (3, 4.5) (c2) {$c_2$};
    \node[main node] at (4.5, 4.5) (d2) {$d_2$};

    \node[main node] at (0, 3.5) (a1) {$a_1$};
    \node[main node] at (1.5, 3.5) (b1) {$b_1$};
    \node[main node] at (3, 3.5) (c1) {$c_1$};
    \node[main node] at (4.5, 3.5) (d1) {$d_1$};

    \draw (a3) -- (b4) -- (c4) -- (d3) -- (c3) -- (b3) -- (a1) -- (b1) -- (c1) -- (d1);
    \draw (c3) -- (b2) -- (a2) -- (b2) -- (c2) -- (d2);

    \draw[->,color=color2,line width=0.8mm] (a2) -- (b2) -- (c3) -- (d3) -- (c4) -- (b4) -- (a3);

    \draw[color=color0,line width=0.8mm] (a1.30) -- (b1.150);
    \draw[color=color0,line width=0.8mm] (b1.30) -- (c1.150);
    \draw[color=color0,line width=0.8mm] (c1.30) -- (d1.150);
    \draw[color=color0,line width=0.8mm] (d1.210) -- (c1.-30);
    \draw[color=color0,line width=0.8mm] (c1.210) -- (b1.-30);
    \draw[->,color=color0,line width=0.8mm] (b1.210) -- (a1.-30);
  \end{tikzpicture}
  \caption{The middle part.}
\end{subfigure}
\hspace{1em}
\begin{subfigure}{0.3\textwidth}
  \begin{tikzpicture}[scale=0.9,thick,main
    node/.style={circle,fill=blue!10,draw,font=\sffamily\bfseries,minimum size=0.6cm,inner sep=0pt}]

    \node[main node] at (0, 6.5) (a3) {$a_3$};
    \node[main node] at (1.5, 6.5) (b4) {$b_4$};
    \node[main node] at (3, 6.5) (c4) {$c_4$};
    \node[main node] at (4.5, 6.5) (d3) {$d_3$};

    \node[main node] at (1.5, 5.5) (b3) {$b_3$};
    \node[main node] at (3, 5.5) (c3) {$c_3$};

    \node[main node] at (0, 4.5) (a2) {$a_2$};
    \node[main node] at (1.5, 4.5) (b2) {$b_2$};
    \node[main node] at (3, 4.5) (c2) {$c_2$};
    \node[main node] at (4.5, 4.5) (d2) {$d_2$};

    \node[main node] at (0, 3.5) (a1) {$a_1$};
    \node[main node] at (1.5, 3.5) (b1) {$b_1$};
    \node[main node] at (3, 3.5) (c1) {$c_1$};
    \node[main node] at (4.5, 3.5) (d1) {$d_1$};

    \draw (a3) -- (b4) -- (c4) -- (d3) -- (c3) -- (b3) -- (a1) -- (b1) -- (c1) -- (d1);
    \draw (c3) -- (b2) -- (a2) -- (b2) -- (c2) -- (d2);

    \draw[->,color=color0,line width=0.8mm] (a1) -- (b3) -- (c3) -- (b2) -- (a2);

    \draw[color=color2,line width=0.8mm] (a3.30) -- (b4.150);
    \draw[color=color2,line width=0.8mm] (b4.30) -- (c4.150);
    \draw[color=color2,line width=0.8mm] (c4.210) -- (b4.-30);
    \draw[->,color=color2,line width=0.8mm] (b4.210) -- (a3.-30);
  \end{tikzpicture}
  \caption{The end part.}
\end{subfigure}
  \caption{Siek\l{}ucki's example and the 2-obstructor for it shown as a pair of paths in $G$ (in three stages).}\label{fig:siek}
\end{figure}

\begin{example}\label{ex:siek}
  This example, along with the description of a 2-obstructor, is presented in~\cite[\S
  3.1]{akh_rep_skop}. The map $f \from G \to H$ first appeared in the proof of Theorem~2.1 in
  \cite{siek}.

  Consider the graphs $G$ and $H$ and the simplicial map $f \from G \to H$ as in~\cref{fig:siek}. As
  shown in the figure, there is a 2-obstructor for $f$ given by the path in $G^{(2)}_f$ defined by
  the two paths in $G$:
  \[
    \begin{aligned}
      \gamma_1 \from a_3 &\to b_4 \to c_4 \to d_3 \to c_3 \to b_3 \to a_1 \to b_1 \to c_1 \to d_1 \to
      c_1 \to \\ &\to b_1 \to a_1 \to b_3 \to c_3 \to b_2 \to a_2 \\
      \gamma_2 \from a_2 &\to b_2 \to c_2 \to d_2 \to c_2 \to b_2 \to a_2 \to b_2 \to c_3 \to d_3
      \to c_4 \to \\ &\to b_4 \to a_3 \to b_4 \to c_4 \to b_4 \to a_3
    \end{aligned}
  \]

  Therefore, $f$ does not lift to an embedding.
\end{example}

\cref{thm:obstructors} provides only the necessary conditions for the existence of a lifting
(see~\cref{ex:giller} and~\cref{ex:nontrivial_gamma}). The next two theorems, however, provide
conditions that are both necessary and sufficient.

\begin{definition}[admissible collection of linear orders]
  Let $f \from K \to L$ be a non-degenerate simplicial map between finite simplicial complexes. For
  each simplex $C \in L$, let $K_C$ be the set of simplices $A \in K$ that map to $C$ under
  $f$. Note that since $f$ is non-degenerate, all the simplices in $K_C$ have the same dimension as
  $C$. Furthermore, let $V(L)$ be the set of vertices of $L$.

  We call a collection $\{ (\prec_v, K_v) \}_{v \in V(L)}$ of linear orders on the sets
  $K_v = f^{-1}(v)$ \emph{admissible}\index{admissible collection of linear orders} if it induces
  linear orders on all the sets $K_C,\ C \in L$.

  Specifically, given a pair of simplices $A, B \in K_C$, let $V_{f}(A, B) \subset V(A) \times V(B)$
  represent all pairs of vertices $(v, w)$ satisfying $f(v) = f(w)$. Then it is required that either
  $v \preceq_{f(v)} w$ holds for all $(v, w) \in V_{f}(A, B)$, or $v \succeq_{f(v)} w$ holds for all
  $(v, w) \in V_{f}(A, B)$.
\end{definition}

\begin{theorem}\label{thm:orders}
  Let $f \from K \to L$ be a non-degenerate simplicial map between finite simplicial complexes. Then
  the piecewise linear map $|f| \from |K| \to |L|$ lifts to an embedding if and only if there exists an
  admissible collection of linear orders on the sets $K_v,\ v \in V(L)$.

  Furthermore, there exists a bijection between the admissible collections of linear orders and the
  isotopy classes of liftings.
\end{theorem}

The reader may refer to~\cite[Proposition 4]{giller},~\cite[Lemme 1.4]{poen_fr}, and~\cite[Theorem
5]{mel2} for variants of the stated theorem.

\begin{proof}
  Let us first, by using a lifting of $|f|$, construct an admissible collection of linear
  orders. Assume there is a lifting $\wt{|f|} \from |K| \to |L| \times \R$. It defines a collection
  of linear orders as follows: for vertices $v, w \in V(K)$ such that $f(v) = f(w)$, we define
  $v \prec_f w$ if $h(v) < h(w)$, where $h = \pr_\R \circ \wt{|f|}$. Let us prove that this
  collection is admissible.

  Suppose it is not. It implies that there exist simplices $A, B \in K$ with the same image $f(A) =
  f(B)$ and two pairs of vertices $f(v) = f(w), f(v') = f(w')$, where $v, v' \in A$, $w, w' \in B$,
  such that $v \prec_f w$ but $v' \succ_f w'$.

  Clearly, we have $f(\{v, v'\}) = f(\{w, w'\}) = \{f(v), f(w)\}$ where $\{a, b\}$ denotes an edge
  with endpoints $a$ and $b$. Let $d \from [0, 1] \to \R$ be the map defined as
  $d(t) = h((1-t) v + t v') - h((1-t) w + t w')$. Since $v \prec_f w$ and $v' \succ_f w'$, we have
  $d(0) = h(v) - h(w) < 0$ and $d(1) > 0$. Therefore, there exists $t' \in [0, 1]$ such that
  $d(t') = 0$. However, since $|f|$ is linear on the simplices of $K$, this implies that
  $\wt{|f|}((1-t')v + t'v') = \wt{|f|}((1-t')w + t'w')$, which contradicts the fact that $\wt{|f|}$
  is an embedding.

  Now suppose we have an admissible collection of linear orders $\{ (\prec_v, K_v) \}_{v \in
    V(L)}$. Let us construct a lifting $\wt{|f|}$ of $|f|$ based on this collection. First, we define a map
  $h \from V(K) \to \N$ as follows: for each $K_v$, we let $h\big|_{K_v}$ be the order isomorphism
  $K_v \to \{1, \dots, |K_v|\}$ that takes the $k$-th vertex in $K_v$ (in ascending order with respect
  to $\prec_v$) to $k$. By extending the map $h$ linearly to the entire $|K|$, we obtain the map
  $h \from |K| \to \R$. Take $\wt{|f|} = |f| \times h$. Now, let us prove that $\wt{|f|}$ is an embedding.

  Suppose it is not. Then there exist points $v, w \in |K|$ such that $\wt{|f|}(v) =
  \wt{|f|}(w)$. Since $v$ and $w$ cannot be vertices of $K$ by the definition of $h$, they lie
  in the interior of some simplices $A$ and $B$ of $K$, respectively, with
  $\dim A = \dim B > 0$.\footnote{In this context, we define the interior of a simplex $\Delta$ as
    $|\Delta| \setminus \bigcup_{A} |A|$, where $A$ ranges over all proper faces of $\Delta$.} The last
  equality follows from the fact that $f$ is non-degenerate and hence injective on simplices, so
  $f(A) = f(B)$ and the simplices $A$, $B$, and $f(A)$ have the same dimension. Let
  $a_0, \dots, a_n$ denote the vertices of $A$, and $b_0, \dots, b_n$ denote the corresponding
  vertices of $B$, such that $f(a_i) = f(b_i)$ for each $i$.

  Now, let $v = \sum_{i=0}^n \alpha_i a_i$ where the $\alpha_i$'s are the barycentric coordinates of $v$ in
  $A$. Since $|f|(v) = |f|(w)$, we have $w = \sum_{i=0}^n \alpha_i b_i$. Therefore, we have $h(v) = \sum_{i=0}^n
  \alpha_i h(a_i) = \sum_{i=0}^n \alpha_i h(b_i) = h(w)$.

  Because $h$ was defined using the admissible collection of linear orders, we have either
  $h(a_i) \leq h(b_i)$ or $h(a_i) \geq h(b_i)$ for all $i$. Moreover, since $A \neq B$, at least one
  of the inequalities is strict. This implies that either $h(v) < h(w)$ or $h(v) > h(w)$, which
  contradicts the assumption that $\wt{|f|}(v) = \wt{|f|}(w)$, where $\wt{|f|} = |f| \times h$.

  To establish the bijection between admissible collections of linear orders and isotopy classes of
  liftings, we observe that each lifting $\wt{|f|}$ is isotopic to the lifting
  $\wt{|f|}\vphantom{f}'$ constructed in the previous step using the admissible collection of linear
  orders induced by $\wt{|f|}$. The isotopy $F \from |K| \times [0, 1] \to |L| \times \R$ is defined
  as $F(x, t) = \left(|f|(x), (1-t)h(x) + t h'(x)\right)$ where $h = \pr_\R \wt{|f|}$ and
  $h' = \pr_\R \wt{|f|'}$. The proof that this is an actual isotopy is left to the reader.
\end{proof}

The next theorem demonstrates that the problem of determining the existence of a lifting for a
piecewise linear non-degenerate map between polyhedra can be reduced to the case of graphs.

\begin{theorem}\label{thm:lifting_sk}
  Let $f \from K \to L$ be a non-degenerate simplicial map. Then the corresponding piecewise linear
  map $|f| \from |K| \to |L|$ lifts to an embedding if and only if
  $\left|f \big|_{\sk_1 K}\right| \from |\sk_1 K| \to |\sk_1 L|$ lifts to an embedding. Here, $\sk_1 K$ and
  $\sk_1 L$ denote the one-dimensional skeletons of $K$ and $L$, respectively. Moreover, there is a
  bijection between the isotopy classes of liftings of $|f|$ and of $\left|f \big|_{|\sk_1 K|}\right|$.
\end{theorem}

\begin{proof}
  Based on~\cref{thm:orders}, it suffices to show that a collection
  ${\{(\prec_v, K_v)\}_{v \in V(K)}}$ of linear orders is admissible with respect to $f$ if and only
  if it is admissible with respect to $f \big|_{\sk_1 K}$.

  It is clear that once it is admissible with respect to $f$, it is admissible with respect to $f
  \big|_{\sk_1 K}$ as well, according to the definition.

  Suppose a collection is admissible with respect to $f \big|_{\sk_1 K}$, but not admissible with respect to
  $f$. This implies that there exist simplices $A, B \in K_C$ and two pairs of vertices
  $f(v) = f(w), f(v') = f(w')$, where $v, v' \in A$, $w, w' \in B$, such that $v \prec_{f(v)} w$ but
  $v' \succ_{f(v')} w'$. The edges $\{v, v'\}, \{w, w'\}$ lie in $\sk_1 K$ and have the same image, but their
  endpoints are in the opposite order. This leads to a contradiction.
\end{proof}

Note that, while studying the existence of a lifting of a piecewise linear map $f$, it suffices to
focus attention only on the subpolyhedron of the preimages of the multiple points of $f$. In view
of~\cref{thm:lifting_sk}, after setting $f'$ to be a restriction of $f$ on the set
$\{ x \in \sk_1 K\ |\ |f^{-1}(f(x))| > 1 \} \subseteq \sk_{1} K$, we have that $f$ lifts if and only
if $f'$ lifts.

For instance, in the case of generic immersions of surfaces into $\R^{3}$, a similar observation
allows us to narrow down the question of liftability of an immersion to the liftability of its
restriction on the preimage of the set of multiple points, see~\cite[Theorem
3.2]{carter_saito_paper}. Note that this restriction is, in fact, a map between graphs.

The next theorem demonstrates that the liftability of a non-degenerate piecewise linear map
$|f| \from |K| \to |L|$ corresponding to a simplicial map $f \from K \to L$, can be encoded as the
satisfiability of a 3-CNF boolean formula, which we will define below.

The key idea is as follows: assuming the covering map $p_2 \from |K^{(2)}_f| \to |\wt{K}^{(2)}|$ is
trivial, each trivialization of it yields two mutually inverse collections of binary relations on
the preimages of the vertices of $L$. Specifically, by selecting a section
$s \from |\wt{K}^{(2)}| \to |K^{(2)}_f|$ of $p_2$, we define the relation as $x\,R\,y$ when $(x, y)$
lies in $s(|\wt{K}^{(2)}|)$. This relation is antisymmetric and irreflexive; therefore, to be a
(strict) order, it only needs to satisfy the transitivity property. Transitivity, in turn, can be
represented by a set of conditions of the form $x\,R\,y \land y\,R\,z \to x\,R\,z$.\footnote{An
  alternative approach to expressing transitivity in terms of the covering maps $p_{k}$'s is to
  find a pair of ``compatible'' trivializations for $p_2 \from |K^{(2)}_f| \to |\wt{K}^{(2)}|$ and
  $p_3 \from |K^{(3)}_f| \to |\wt{K}^{(3)}|$, refer to~\cite[Theorem 5]{mel2} for further details.}

More precisely, let us define a boolean formula $\Gamma_f$. Denote by $\CC$ the set of connected
components of $|K^{(2)}_f|$. The involution $\tau$ on $|K^{(2)}_f|$ induces an involution on
$\CC$. Let us assume that the covering map $p_2$ is trivial, therefore, the latter involution is
fixed-point free.

Let $O_1, O_2, \dots, O_n$ be the orbits of the involution on $\CC$. For each orbit $O_i$, pick
an element $C_i \in O_i$ in the orbit, so $O_i = \{ C_i, \tau(C_i) \}$, and associate a boolean variable $x_i$
with $C_i$. Next, we associate the literal $\neg x_i$ with each $\tau(C_i)$. As a result, each
connected component $C$ in $\CC$ is associated with either a variable or its negation.
We refer to the literal corresponding to $C$ as $\alpha_C$.

Let $T$ be the set of ordered triples $(C, D, E)$ of connected components, with $C \neq E$ and
$D \neq E$, such that there exist three vertices $a, b, c \in V(K)$ satisfying $(a, b) \in C$,
$(b, c) \in D$, $(a, c) \in E$. Let $\Gamma_f$ be the boolean function
\[
  \Gamma_f(x_1, \dots, x_n) = \bigwedge_{(C, D, E) \in T} ((\alpha_C \land \alpha_D) \to \alpha_E)
                          = \bigwedge_{(C, D, E) \in T} (\neg \alpha_C \lor \neg \alpha_D \lor \alpha_E)
\]

Note that, aside from reordering the clauses, literals, and renaming variables, the only source of
variability in the form of $\Gamma_f$ arises from the initial choice of representatives for the
orbits $O_{i}$. However, it is easy to check that once $\Gamma_f$ contains a clause
$\neg \alpha_C \lor \neg \alpha_D \lor \alpha_E$, then it must also contain the clause
$\alpha_C \lor \alpha_D \lor \neg \alpha_E$.\footnote{After double negation elimination, that is
  $\neg \neg x = x$.} This is because, once we have $(a, b) \in C$, $(b, c) \in D$, and
$(a, c) \in E$, we have $(c, b) \in \tau(D)$, $(b, a) \in \tau(C)$, and $(c, a) \in
\tau(E)$.

Thus, the form of $\Gamma_{f}$ does not actually depend on the choice of representatives, and,
therefore, $\Gamma_f$ is uniquely defined up to reordering the clauses and the literals within
clauses, and renaming the variables.

\begin{theorem}\label{thm:gamma_cond}
  Let $f \from K \to L$ be a non-degenerate simplicial map. Then the corresponding piecewise linear map
  $|f| \from |K| \to |L|$ lifts to an embedding if and only if the following conditions hold:
  \begin{enumerate}
  \item\label{thm:gamma_cond:p2_trivial} the covering map $p_2 \from |K^{(2)}_f| \to |\wt{K}^{(2)}_f|$ is trivial,
  \item\label{thm:gamma_cond:gamma} $\Gamma_f$ is satisfiable.
  \end{enumerate}

  Furthermore, there exists a bijection between the assignments that satisfy $\Gamma_f$ and the isotopy classes of liftings.
\end{theorem}

\begin{proof}
  Let us assume that there exists a lifting $\wt{|f|} \from |K| \to |L| \times \R$ of
  $|f|$. By~\cref{thm:obstructors}, the covering map $p_2$ is trivial. Moreover, as seen in the
  proof of~\cref{thm:obstructors}, we can obtain a continuous equivariant map $H$ from $|K^{(2)}_f|$
  to $S_2$ by composing $h \from |K^{(2)}_f| \to \Conf_2(\R)$ defined as $h((x, y)) = \left(\pr_\R
    \circ \wt{|f|}(x), \pr_\R \circ \wt{|f|}(y)\right)$, with the retraction $\Conf_2(\R) \to
  S_2$.

  Let us define the variable assignment by putting $x_i = 1$ if $H(C_{i}) = \id \in S_2$, and
  $x_{i} = 0$ otherwise. This assignment is well-defined because $H$ is constant on the connected
  components of $|K^{(2)}_f|$. Moreover, under this assignment, $\alpha_{C} = 1$ if and
  only if $H(C) = \id$.

  Let us consider an implication $(\alpha_C \land \alpha_D) \to \alpha_E$ in $\Gamma_f$. According
  to the definition of $\Gamma_f$, there exist pairs $(a, b) \in C$, $(b, c) \in D$, and
  $(a, c) \in E$. Once $\alpha_C = 1$ and $\alpha_D = 1$, we have $\alpha_E = 1$ by the transitivity
  of $\prec_{f(a)}$. Therefore, the defined assignment satisfies the disjunction, and consequently,
  it satisfies the entire formula $\Gamma_f$.

  Let us prove the reverse implication. Assume we have an assignment
  $\psi \from \{ x_{i} \}_{i=1}^{n} \to \{0, 1\}$ that satisfies $\Gamma_f$. We denote by
  $\psi(\alpha_C)$ the value of $\alpha_C$ under the assignment $\psi$. Now, we define a binary
  relation $R$ on $V(K^{(2)}_f)$ as follows: $a\, R\, b$ if $\psi(\alpha_C) = 1$ for the connected
  component $C$ containing the pair $(a, b)$, if it exists.

  It is easy to see that the relation $R$ is antisymmetric and irreflexive. Moreover, it is
  transitive. Indeed, suppose that pairs $(a, b)$ and $(b, c)$ lie in connected components $C$ and
  $D$, respectively, and $a\, R\, b$ and $b\, R\, c$, or, in other words,
  $\psi(\alpha_C) = \psi(\alpha_D) = 1$. Denote by $E$ the connected component containing the pair
  $(a, c)$. If $E = C$ or $E = D$, then $\psi(\alpha_E) = \psi(\alpha_C) = 1$ or
  $\psi(\alpha_E) = \psi(\alpha_D) = 1$, respectively, hence $a\, R\, c$. Otherwise, we have a disjunction
  $\neg \alpha_C \lor \neg \alpha_D \lor \alpha_E$ in $\Gamma_f$ confirming that $\psi(\alpha_E) = 1$ and
  hence $a\, R\, c$.  Thus, $R$ defines a collection of linear orders on the sets $K_v$ by putting
  $a \prec_{f(a)} b$ if $a\, R\, b$.

  Note that if $(a, b), (c, d) \in V(K^{(2)}_{f})$ lie in the same connected component of
  $K^{(2)}_{f}$, then, according to the definition of $R$, either $a \prec_{f(a)} b$ and
  $c \prec_{f(c)} d$, or $a \succ_{f(a)} b$ and $c \succ_{f(c)} d$.

  Next, let us prove that this collection is admissible. To do this, we need to show that, given
  simplices $A$ and $B$ with the same image, we have either $a \prec_{f(a)} b$ or $a \succ_{f(a)} b$
  simultaneously for all pairs of distinct vertices $(a, b) \in V(A) \times V(B)$ with
  $f(a) = f(b)$. However, all such pairs lie in the same connected component of $K^{(2)}_{f}$, which
  provides the desired result.

  Thus, we have constructed an admissible collection of linear orders on the sets $K_v$, and,
  by~\cref{thm:orders}, it defines a lifting of $f$.

  Note that, in the above proof, we established a bijective correspondence between the assignments
  satisfying $\Gamma_{f}$ and the collections of admissible orders. Therefore, by~\cref{thm:orders},
  this establishes a bijection between the assignments and the isotopy classes of liftings. This
  observation completes the proof.
\end{proof}

\begin{remark}
  Both~\cref{thm:orders} and~\cref{thm:gamma_cond} are proven under the assumption that $f$ is a
  simplicial map. However, it is easy to see that the statements of these theorems remain true if we
  replace simplicial complexes with multigraphs, and simplicial maps with homomorphisms of multigraphs.
\end{remark}

\begin{figure}
  \centering
  \includegraphics[width=.99\linewidth,height=15cm,keepaspectratio]{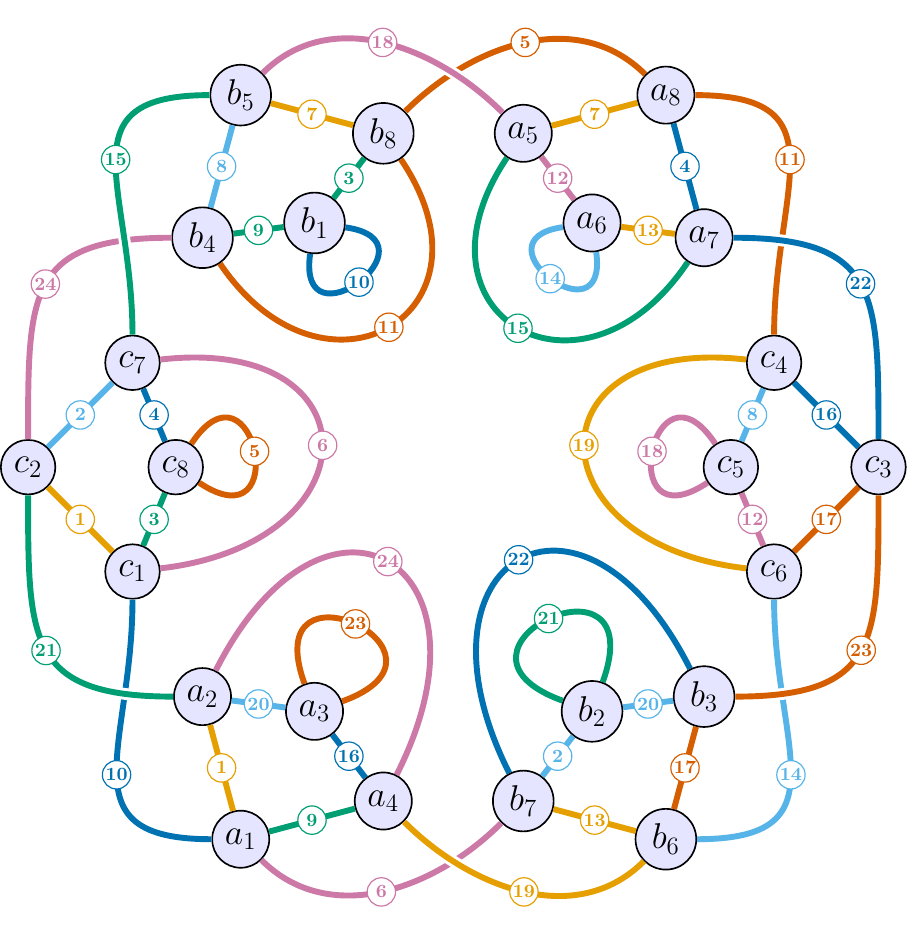}
  \caption{Giller's example: graph of multiple point preimages. Vertices with the same subscript map
    to the same vertex, and identically labelled edges have the same image. In other words,
    $f(x_i) = f(y_j)$ if and only if $i = j$, and $f(e) = f(g)$ if and only if $e$ and $g$ share the
    same label.}\label{fig:giller:main}
\end{figure}

\begin{figure}
  \centering
  \includegraphics[width=.99\linewidth,height=15cm,keepaspectratio]{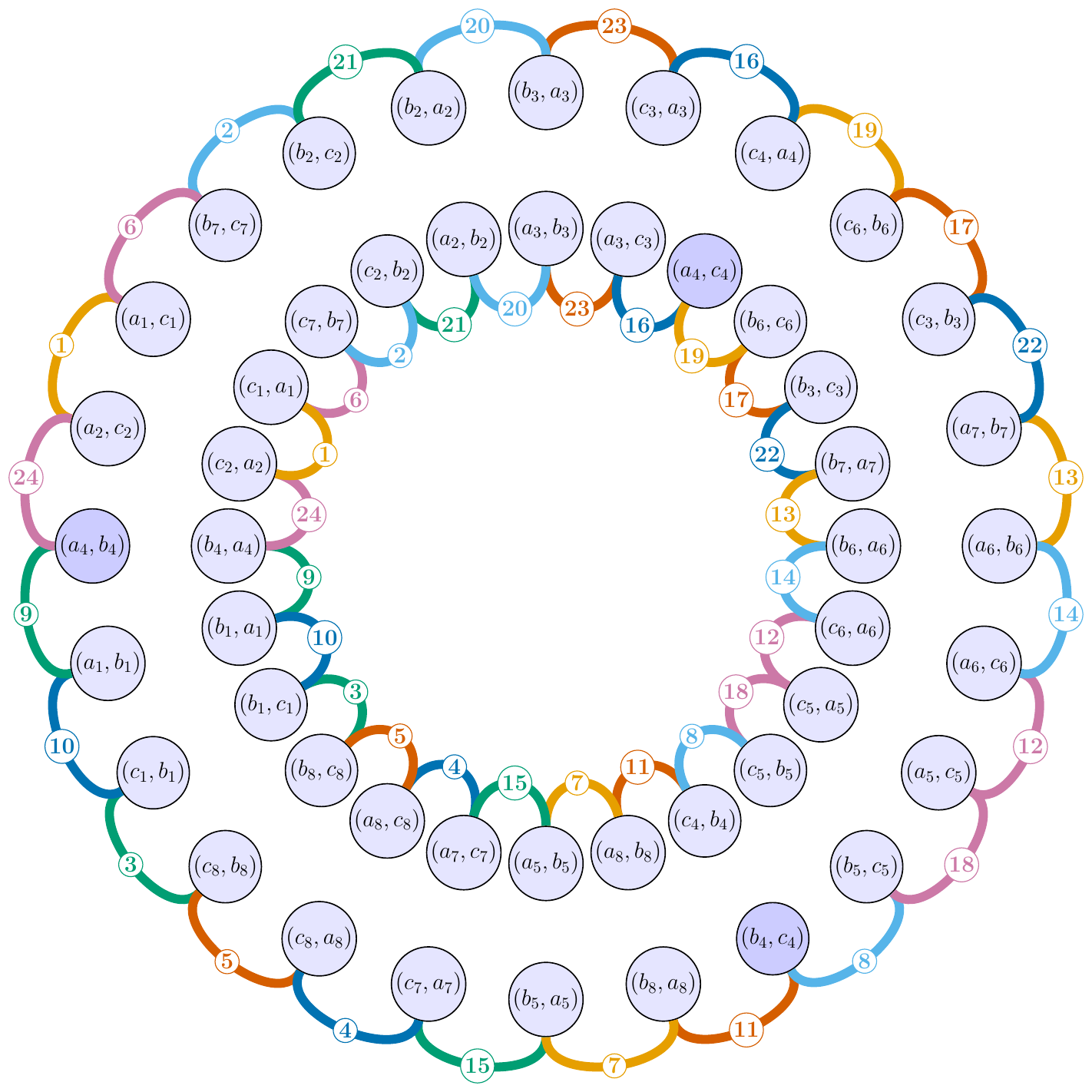}
  \caption{Giller's example: the graph $G^{(2)}_f$.}\label{fig:giller:double}
\end{figure}

\begin{example}[Giller's example, see~\cite{giller}]\label{ex:giller}
  The provided example, known as Giller's example, illustrates an immersion of the 2-sphere $S^2$ into
  3-dimensional Euclidean space $\mathbb{R}^3$ that cannot be lifted to an embedding. The example is
  discussed in the paper by Giller~\cite{giller}, and we refer the reader to it for details. We
  focus only on the graph of the multiple point preimages to understand its liftability. We denote
  by $f$ the restriction of the immersion on this graph.

  The multigraph of the multiple point preimages is shown in~\cref{fig:giller:main}.

  Additionally,~\cref{fig:giller:double} shows the corresponding graph $G^{(2)}_f$, which is the
  union of two circles. Note that the covering map $p_2$ is trivial in this case.

  However, one can see that the pairs $(a_4, b_4)$ and $(b_4, c_4)$ lie in the outer circle
  of~\cref{fig:giller:double}, while $(a_4, c_4)$ lies in the inner circle. As a result, $\Gamma_f$
  contains clauses $(x \lor x \lor x) \land (\neg x \lor \neg x \lor \neg x)$, where $x$ is a
  variable corresponding to the inner circle. Since this subformula is equivalent to
  $x \land \neg x$, the entire formula $\Gamma_f$ is not satisfiable. Therefore, both $f$ and the
  initial immersion do not lift to an embedding.
\end{example}

Let us now assume that for an arbitrary map $f$ the formula $\Gamma_f$ contains a clause
corresponding to a triple $(C, C, E)$. Thus, it contains a clause corresponding to a triple
$(\tau(C), \tau(C), \tau(E))$ as well, and, therefore,
\[
  \Gamma_f = (\neg \alpha_C \lor \neg \alpha_C \lor \alpha_E) \land (\alpha_C \lor \alpha_C \lor \neg
  \alpha_E) \land \Gamma' = (\alpha_C = \alpha_E) \land \Gamma'
\]

This allows us to simplify the formula $\Gamma_f$ by ``gluing'' connected components $C$ and $E$
(that is, defining an equivalence relation on the points of $|K^{(2)}_f|$ that is coarser than the
one defined by the partition into connected components) once we meet a triple of pairs
$(a, b) \in C$, $(b, c) \in C$, and $(a, c) \in D$. This idea is implemented in~\cite{poen,poen_fr};
we will discuss it, and disprove some of the results stated in these papers, in the next section.

\FloatBarrier

\section{Liftings of smooth immersions}

In this section, we explore connections between lifting maps between graphs and lifting generic
immersions between smooth manifolds. We provide a counterexample to a result stated
in~\cite{poen,poen_fr}. Additionally, we show that for each $\Gamma$ that potentially could be a
formula from~\cref{thm:gamma_cond}, there exists a generic smooth immersion
$f \from X \looparrowright Y$ of a surface $X$ with boundary into a 3-manifold $Y$ with boundary
such that the restriction $f'$ of $f$ onto the preimage of the multiple point set forms a map
between graphs with $\Gamma_{f'} = \Gamma$.

We will use the word ``generic'' in the same sense in which it is used in~\cite{poen}. Specifically,
we say that $f \from X \looparrowright Y$ is \emph{generic}\index{generic immersion} when
\begin{enumerate}
\item $f^{-1}(\partial Y) = \partial X$,
\item for $y \in Y \setminus \partial Y$ and $f^{-1}(y) = \{x_i\}_{i=1}^{n}$, the subspaces
  $df_{x_i}(T_{x_i}X)$ are in general position in $T_y Y$, that is
  \[
  \codim \bigcap_{i=1}^{n} df_{x_i}(T_{x_i}X) = \sum_{i=1}^{n} \codim df_{x_i}(T_{x_i}X),
  \]
\item for $y \in \partial Y$ and $f^{-1}(y) = \{x_i\}_{i=1}^{n}$, the subspaces $T_y \partial Y$ and
  $df_{x_i}(T_{x_i} X)$ are in general position in $T_y Y$.
\end{enumerate}

Let $f \from X \looparrowright Y$ be a generic immersion. In \cite{poen} and \cite[Th\'{e}or\`{e}me
1]{poen_fr}, Po\'{e}naru claimed that for the existence of a lifting $\wt{f} \from X \to Y \times
\R$ to an embedding, it is necessary and sufficient that $\mu_2(f) = \nu_3(f) = 0$.

Here, the condition $\mu_2(f) = 0$ is equivalent to the triviality of the covering
$P_2 \from \Conf_f(X, 2) \to \wt{\Conf}_f(X, 2)$, where $\Conf_f(X, 2)$ represents the configuration
space of ordered pairs of distinct points with the same image under $f$, and $\wt\Conf_f(X, 2)$ is
its unordered version. For non-degenerate piecewise linear maps between graphs $K$ and $L$, this
condition is equivalent to the first condition in~\cref{thm:gamma_cond}. Indeed, both
$\Conf_f(X, 2)$ and $\wt\Conf_f(X, 2)$ equivariantly deformation retract onto $|K^{(2)}_f|$ and
$|\wt{K}^{(2)}_f|$, respectively. Therefore, the triviality of $P_2$ is equivalent to the triviality
of $p_2 \from |K^{(2)}_f| \to |\wt{K}^{(2)}_f|$.

Now, let us consider the condition $\nu_3(f) = 0$. Following Po\'{e}naru's approach, we define a set
$S_{2,3} \subseteq \Conf_f(X, 2) \times \wt{\Conf}_f(X, 3)$ that includes points
$(a, b) \times \{c, d, e\}$ where $\{a, b\} \subset \{c, d, e\}$. On $S_{2,3}$, we introduce a
collection $\TT$ of all equivalence relations satisfying the following properties (denoting
unordered triples as $T_1$ and $T_2$ for simplicity):
\begin{enumerate}
\item if $(a, b) \times T_1 \sim (c, d) \times T_2$, then $(b, a) \times T_1 \sim (d, c) \times T_2$,
\item if the points $(a, b), (c, d)$ lie in the same connected component of $\Conf_f(X, 2)$, then
  $(a, b) \times T_1 \sim (c, d) \times T_2$,
\item once $(a, b) \times \{a, b, c\} \sim (b, c) \times \{a, b, c\}$, we have $(a, b) \times \{a, b,
  c\} \sim (b, c) \times \{a, b, c\} \sim (a, c) \times \{a, b, c\}$.
\end{enumerate}

Since the intersection of two equivalence relations in $\TT$ belongs to $\TT$, the set $\TT$ forms a
lower subsemilattice of the lattice $\Eq(S_{2,3})$ consisting of all equivalence relations on the
set $S_{2,3}$. Therefore, let $\sim_\nu$ denote the smallest (finest) equivalence relation in
$\mathbb{T}$. We say that $\nu_3(f) = 0$ if there are no points such that
$(a, b) \times \{a, b, c\} \sim_\nu (b, c) \times \{a, b, c\} \sim_\nu (c, a) \times \{a, b, c\}$.

Considering $(a, b) \times T_1 \sim (a, b) \times T_2$ for any $\sim$ from $\TT$ and for any points
$T_1, T_2 \in \wt{\Conf}_f(X, 3)$ with $\{a, b\} \subset T_{1}, T_{2}$, each equivalence relation
$\sim$ from $\TT$ defines the equivalence relation $\psi(\sim)$ on the set $\Conf_f(X, 2)$ by taking
$(a, b)\ \psi(\sim)\ (c, d)$ if $(a, b) \times T_1 \sim (c, d) \times T_2$ for some
$T_1, T_2 \in \wt{\Conf}_f(X, 3)$. One can check that
$\psi \from \TT \hookrightarrow \Eq(\Conf_f(X, 2))$ defines an embedding of the semilattice $\TT$
into the lattice $\Eq(\Conf_f(X, 2))$ consisting of all the equivalence relations on
$\Conf_f(X, 2)$.

It can be verified that $\psi(\TT) \subset \Eq(\Conf_f(X, 2))$ can be defined using the three
conditions defining $\TT$, where instead of the points of the form $(x, y) \times T$, we consider
pairs $(x, y)$. Therefore, the second condition, saying that $(a, b) \sim (c, d)$ whenever $(a, b)$ and $(c,
d)$ lie in the same connected component of $\Conf_f(X, 2)$, ensures that all equivalence relations
in $\psi(\TT)$ are coarser than the one defined by the partition of $\Conf_f(X, 2)$ into connected
components.

Moreover, as demonstrated in the proof of~\cite[Lemme 1.2]{poen_fr}, the equivalence relation
$\psi(\sim_\nu)$ (and, thus, $\sim_\nu$) can be iteratively constructed as follows: we start with
the equivalence relation on~$\Conf_{f}(X, 2)$ defined by the partition into connected
components. We then sequentially join connected components $C$ and $D$ whenever we meet a triple of
pairs $(a, b), (b, c) \in C$, and $(a, c) \in D$.

As discussed in the end of the preceding section, this process corresponds to simplifying the
formula $\Gamma_f$ by replacing terms of the form
$(\neg \alpha_C \lor \neg \alpha_C \lor \alpha_D) \land (\alpha_C \lor \alpha_C \lor \neg \alpha_D)$
with equalities $\alpha_C = \alpha_D$. In view of~\cref{thm:gamma_cond}, Po\'{e}naru's claim is
equivalent to the fact that, for generic immersions between smooth manifolds, the part of the
formula remaining after simplifications is always satisfiable.\footnote{This remark should be
  understood as an informal comment, since formally, the formulas $\Gamma_f$ are defined only for
  non-degenerate simplicial maps; however, the fact that the problem of lifting to embeddings of
  generic immersions between manifolds reduces to the problem of lifting maps between graphs
  (see~\cite[pages 1.21--1.22]{poen_fr} and~\cite[Theorem~4.6]{carter_saito_book}) makes this
  comment meaningful.} The next example together with the next theorem show that the latter
statement is not true.

\begin{figure}[h!]
  \centering
  \begin{subfigure}{1.0\linewidth}
    \centering
    \includegraphics[width=.99\linewidth,height=12cm,keepaspectratio]{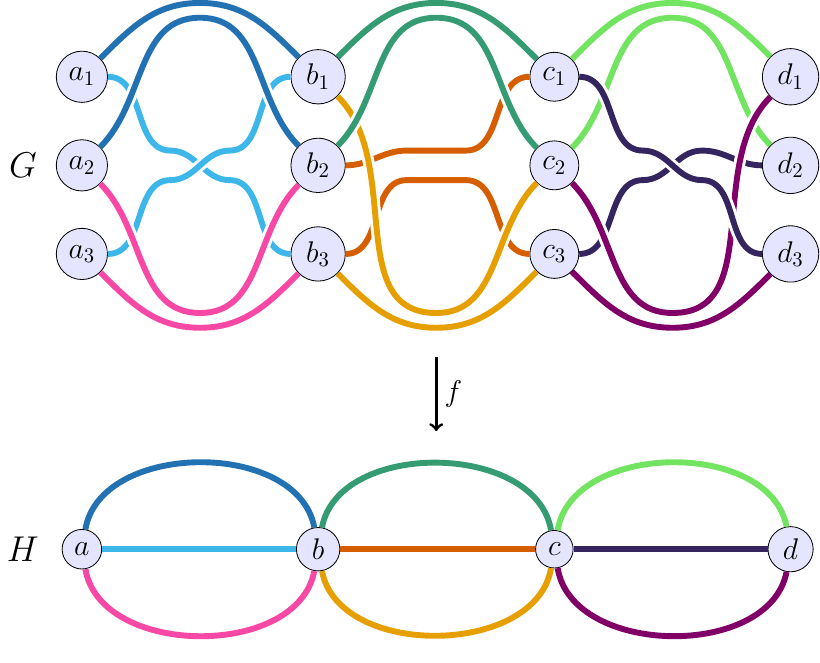}
    \caption{The map $f \from G \to H$.}\label{fig:nontrivial_gamma:f}
  \end{subfigure}

  \vspace{1cm}
  \begin{subfigure}{1.0\linewidth}
    \centering
    \includegraphics[width=.99\linewidth,height=12cm,keepaspectratio]{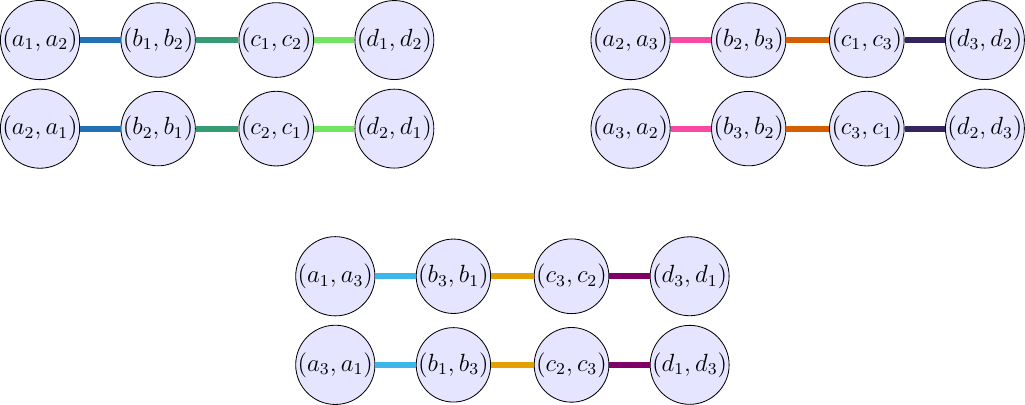}
    \caption{The graph $G^{(2)}_f$.}\label{fig:nontrivial_gamma:K2}
  \end{subfigure}
  \caption[An example of a map with ``nontrivial'' $\Gamma_f$.]{An example of a map with
    ``nontrivial'' $\Gamma_f$. The edges of the graph $G$ and their images in the graph $H$ are coloured
    identically, the same holds for edges of $G^{(2)}_f$ and edges of $G$ corresponding to
    them. Furthermore, edges of $G$ with the same image are placed closely together, and the
    ``upper'' pairs of edges of $G$ map to the ``upper'' edges of $H$; the same holds for the
    ``middle'' and the ``lower'' pairs.}\label{fig:nontrivial_gamma}
\end{figure}

\begin{example}\label{ex:nontrivial_gamma}
  Let us consider the multigraph homomorphism shown in~\cref{fig:nontrivial_gamma:f}. The graph
  $G^{(2)}_f$ (\cref{fig:nontrivial_gamma:K2}) consists of 6 connected components that split into
  pairs with respect to the involution. Consequently, $p_2 \from |G^{(2)}_f| \to |\wt{G}^{(2)}_f|$
  is trivial.

  Note that for each triple $(p_1, p_2, p_3)$, where $p$ stands for $a$, $b$, $c$, or $d$, all the
  pairs $(p_i, p_j), i \neq j$ lie in different connected components of $G^{(2)}_f$. Each such
  triple defines two disjunctions in $\Gamma_f$.\footnote{Actually, each triple defines 6
    disjunctions, one for each permutation of three points. However, after dropping equal
    disjunctions up to reordering terms, only two disjunctions remain.} Given that, we can express
  $\Gamma_f$ as follows:
  \[
    \begin{aligned}
      \Gamma_f(x_1, x_2, x_3) =\ &((x_1 \wedge x_2) \to x_3) \wedge ((\neg x_1 \wedge \neg x_2) \to \neg x_3)\ \wedge \\
                              &((x_1 \wedge x_2) \to \neg x_3) \wedge ((\neg x_1 \wedge \neg x_2) \to x_3)\ \wedge \\
                              &((x_1 \wedge \neg x_3) \to x_2) \wedge ((\neg x_1 \wedge x_3) \to \neg x_2)\ \wedge \\
                              &((x_1 \wedge \neg x_2) \to \neg x_3) \wedge ((\neg x_1 \wedge x_2) \to x_3)
    \end{aligned}
  \]
  Here $x_1$ represents the connected component containing $(a_1, a_2)$, $x_2$ represents the
  connected component containing $(a_2, a_3)$, and $x_3$ represents the connected component containing $(a_1, a_3)$.

  The first four clauses imply that either $x_1 = 1$ and $x_2 = 0$, or $x_1 = 0$ and $x_2 = 1$. In
  the former case, the next-to-last disjunction yields $x_3 = 0$, which contradicts
  $(x_1 \wedge \neg x_3) \to x_2$ as $x_2 = 0$. In the latter case, the last disjunction implies
  $x_3 = 1$, which contradicts $(\neg x_1 \wedge x_3) \to \neg x_2$ as $x_2 = 1$. Consequently,
  $\Gamma_f$ is not satisfiable, and, as a result, $f$ does not lift to an embedding.

  In spite of this, it can be seen that in the formula $\Gamma_f$, there are no clauses in which
  literals are repeated. In light of the comments above, this means that equivalence classes under
  $\sim_\nu$, defined as it has been done for generic immersions before, coincide with the connected
  components of $|G^{(2)}_f|$.
\end{example}

Now, using the provided example, we will construct a counterexample to the claim of \cite[Th\'{e}or\`{e}me
1]{poen_fr} showing that its assumptions need to be strengthened in the sense of~\cref{thm:gamma_cond}:

\begin{theorem}\label{thm:counterexample}
  There exists a smooth generic immersion $j \from T \looparrowright Y$ from a compact surface $T$
  to a compact 3-manifold $Y$ that satisfies the conditions $\mu_2(j) = \nu_3(j) = 0$, but does not lift
  to an embedding.
\end{theorem}
\begin{proof}
  Let $f \from G \to H$ be the map from~\cref{ex:nontrivial_gamma}. We add three multiple edges
  $e_1, e_2, e_3$ to $H$, going from $a$ to $d$. We also add six edges to $G$ and define $f$ on
  these edges as follows:
  \begin{enumerate}
  \item $f(\{a_1, d_1\}) = f(\{a_2, d_2\}) = e_1$,
  \item $f(\{a_2, d_3\}) = f(\{a_3, d_2\}) = e_2$,
  \item $f(\{a_1, d_3\}) = f(\{a_3, d_1\}) = e_3$.
  \end{enumerate}

  Let us denote the resulting graphs by the same letters $G$ and $H$.

  It can be observed that the covering map $p_2$ for the resulting map $f \from G \to H$ remains
  trivial; moreover, the formula $\Gamma_f$ also coincides with the formula of the original
  map.\footnote{In fact, we have simply replaced the segments that make up the graph $G^{(2)}_f$ by
    joining their endpoints to form circles.}

  Now let us construct a surface $S$ with boundary immersed into a handlebody $B$.

  We represent each vertex $a, b, c, d$ of $H$ as an intersection of three squares with equal sides
  parallel to the coordinate planes, and label each square by the number from one to three. We
  denote the $i$-th square in the intersection corresponding to the vertex $x$ by $X_i$. For
  example, the vertex $a \in V(H)$ is represented by the intersection of three squares $A_1$, $A_2$,
  and $A_3$. Furthermore, each square $X_i$ corresponds to a vertex $x_i$ of $G$. For instance, the
  squares $A_1$, $A_2$, and $A_3$ correspond to the vertices $a_1$, $a_2$, and $a_3$, respectively.

  Thus, each vertex $x \in V(H)$ corresponds to a triple point $X_1 \cap X_2 \cap X_3$. From each
  such triple point, six half-edges start: two for each intersection $X_1 \cap X_2$, $X_1 \cap X_3$,
  and $X_2 \cap X_3$.

  \begin{figure}[h!]
    \centering
    \vspace{2em}
    \includegraphics[width=.99\linewidth,height=10cm,keepaspectratio]{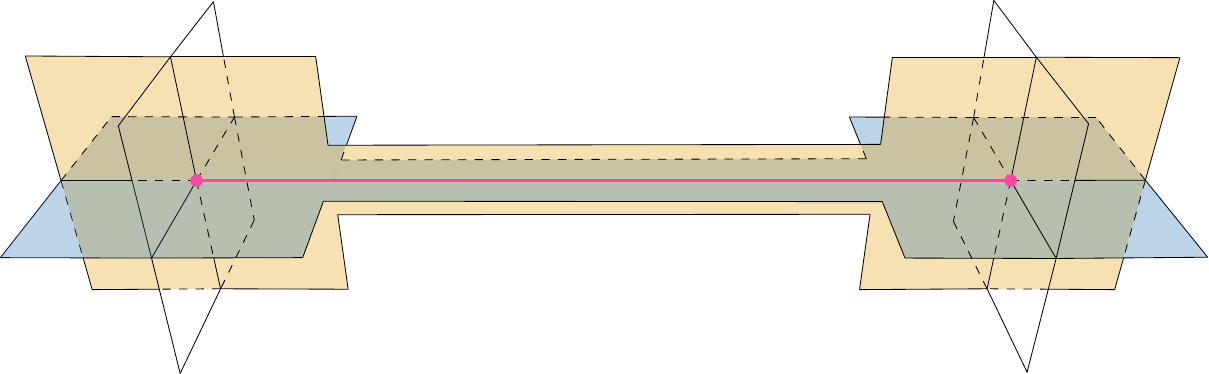}
    \caption{Two triples of intersecting squares, representing triple points, and an X-ribbon
      between them, representing a double edge. The preimages of the pink double edge are an edge lying
      on the blue squares and the blue ribbon between them, and an edge lying on the orange squares
      and the orange ribbon.}\label{fig:squares_ribbon}
  \end{figure}

  For each unordered pair of edges $\{x_i, y_j\}, \{x_k, y_l\} \in E(G)$ with the same image, where
  $f(x_i) = f(x_k) = x$ and $f(y_j) = f(y_l) = y$, we perform the following procedure.

  We choose one of the half-edges from $X_i \cap X_k$ and one of the half-edges lying in
  $Y_j \cap Y_l$ that are currently ``free'' (that is, not chosen in previous steps). Then, we
  connect these half-edges by an X-shaped ribbon,\footnote{That is, the product of two intersecting
    segments (forming the letter X) with another segment.} so that one plane of the X-shaped ribbon connects the side of
  the square $X_i$ with the side of the square $Y_j$, and the other one connects the sides of the
  squares $X_k$ and $Y_l$, see~\cref{fig:squares_ribbon}.

  Note that each vertex of $\wt{G}^{(2)}_f$ has degree two. This ensures that there is always
  a free half-edge to choose. Moreover, after processing all the pairs of edges, each half-edge
  belongs to some X-shaped ribbon.

  Observe that each X-shaped ribbon could be twisted while connecting. Additionally, when performing
  the steps described above, there is flexibility in choosing which half-edges to connect with each
  other. However, in our construction, this choice is not crucial; hence, we select free half-edges
  and twist the ribbons in an arbitrary way.

  Next, we take small 3-balls around each triple point (that is, the points $X_1 \cap X_2 \cap X_3$
  inside the 3-intersections of squares) and small 3-tubes around each curve of double points (that
  is, half-edges connected by the central curve of an X-ribbon). This construction gives us a
  handlebody $B$. The intersections of squares connected by the X-shaped ribbons within $B$ form a
  generic surface $S$ immersed into $B$. Denote the immersion by $g$. Note that the multiple points
  set of $g$ is the graph $H$ embedded into the interior of $B$, and its preimage is the graph $G$
  embedded into $S$. The boundary of $S$ is a circle embedded into the boundary of $B$.

  Furthermore, $g\big|_{|G|}$ acts exactly as $f$ on the curves representing the edges of $G$
  embedded into $S$: $g$ takes the curve between points $x_i$ and $y_j$ to the curve between points
  $f(x_i)$ and $f(y_j)$.

  To obtain a closed 3-manifold $Y$, we take the double of $B$, which is two copies of $B$ glued
  along their common boundary. The double of $S$, denoted as $T$, is then immersed into $Y$ by the
  immersion $j$ induced by $g \sqcup g \from S \sqcup S \looparrowright B \sqcup B$. Clearly, the
  immersion $j$ is in general position. The multiple point set of $j$ is a graph $H \sqcup H$
  embedded into $Y$, and its preimage is a graph $G \sqcup G$ embedded into $T$. Moreover, it is
  clear that $j\big|_{|G| \sqcup |G|}$ acts on the curves representing the edges of $G \sqcup G$ as
  $f \sqcup f$.

  Let us note that the set $S_{2,3}$ in the definition of $\sim_\nu$ for the map $g$ consists of 24
  points. Each point of $S_{2,3}$ represents an ordered pair of vertices in $G$ with the same image
  under $g$, or, equivalently, under $f$. In other words, the set $S_{2,3}$ coincides with the set
  of vertices of the graph $G_f^{(2)}$. Furthermore, it is straightforward to verify that the
  equivalence relation defined by the partition of the vertices of $G_f^{(2)}$ into connected
  components satisfies all the properties in the definition of $\sim_\nu$; therefore, these
  equivalence relations coincide. Given that, it becomes apparent that $\nu_3(g) = 0$. It is also
  evident that $\mu_2(g) = 0$ since $p_2 \from |G_f^{(2)}| \to |\wt{G}_f^{(2)}|$ is trivial.

  As all the multiple points of $g$ lie in the interior of $B$, the set $S_{2,3}$ and the relation
  $\sim_\nu$ for $j$ coincide with those of $g \sqcup g$. Therefore, since
  $\nu_3(g) = \mu_2(g) = 0$, we have $\nu_3(j) = \mu_2(j) = 0$.

  If $j$ lifts to an embedding then so does $f$. However, as shown in~\cref{ex:nontrivial_gamma},
  the formula $\Gamma_f$ is not satisfiable, and consequently, $f$ does not lift to an embedding.
\end{proof}

It can be observed that~\cref{ex:nontrivial_gamma} also provides a
counterexample to~\cite[Lemme 1.7]{poen_fr}. One can notice that the preimage of the cycle
shown in~\cite[Fig. 1.3]{poen_fr} does not necessarily have to be the two cycles as
shown in~\cite[Fig. 1.4]{poen_fr}, as the paper claims, which is where the proof fails.

It is worth noting that the mentioned mistake only affects~\cite[Lemme 1.7]{poen_fr}, which states
that a ``labyrinth''\index{labyrinth} (that is, a map $f_j$ between graphs, obtained using a given
immersion $j$ of manifolds; see~\cite[Definition 3]{poen_fr} for details) lifts if
$\mu_2(f_j) = \nu_3(f_j) = 0$. However, the reduction from the considered case of immersions of
manifolds to maps between graphs is itself correct. Thus, the immersion $j$ lifts if and only if
$f_j$ lifts, or, equivalently, the covering map $p_2$ for $f_f$ is trivial, and $\Gamma_{f_j}$ is
satisfiable. Our next goal is to show that these conditions cannot be weakened in the general case:
we will demonstrate that for any suitable boolean formula $\Gamma$ it is possible to construct a
generic immersion $j$ with $\Gamma_{f_j} = \Gamma$.

Notice that any $\Gamma_f$, as defined in the previous section, satisfies the following condition:
once it contains a clause of the form $(\alpha \land \beta) \to \gamma$, it also contains a clause
$(\neg \alpha \land \neg \beta) \to \neg \gamma$, up to double negation elimination. Additionally,
let us note that the formulas that cannot be simplified by the algorithm described at the end of
the previous section are those where, for each clause $(\alpha \land \beta) \to \gamma$, all
variables constituting $\alpha$, $\beta$, and $\gamma$ are distinct.

The next theorem shows that any $\Gamma$ of the form described above can be realised as $\Gamma_f$
of a multigraph homomorphism $f$ induced by a generic immersion of a surface into a handlebody, similar
to what we have constructed in the previous theorem.

\begin{theorem}\label{thm:realisation_gamma}
  Let $\Gamma(x_1, \dots, x_n)$ be a boolean formula of the form
  \[
    \Gamma(x_1, \dots, x_n) = \bigwedge_{j=1}^m \left((\alpha_j \wedge \beta_j) \to \gamma_j\right)
  \]
  where $\alpha_j$, $\beta_j$, and $\alpha_j$ are variables or their negations. Assume that for every $j$
  \begin{enumerate}
  \item the variables in the literals $\alpha_j$, $\beta_j$, and $\gamma_j$ are distinct,
  \item there exists $j'$ such that $\alpha_{j'} = \neg \alpha_j$, $\beta_{j'} = \neg \beta_{j}$, and $\gamma_{j'}
    = \neg \gamma_j$ (after double negation elimination).
  \end{enumerate}

  Then there exists a generic immersion $g \from S \looparrowright B$ of a surface $S$ with
  boundary into a handlebody $B$ such that the restriction of $g$ onto the preimage of the set of
  multiple points induces a homomorphism $f \from G \to H$ of multigraphs. Additionally, the formula
  $\Gamma_f$ is equivalent to $\Gamma$, differing only in variable renaming, the presence of
  duplicate clauses, and reordering of clauses and literals within clauses.
\end{theorem}

To prove the theorem, we will use the following lemma, which says, informally, that for each
$\Gamma$ as in~\cref{thm:realisation_gamma}, we can construct a multigraph homomorphism
$f \from G \to H$, whose $\Gamma_f$ is equivalent to $\Gamma$, and $f$ is similar to that we used in
the construction of~\cref{thm:counterexample} which allows us, following this construction, to
obtain the needed immersion $g \from S \looparrowright B$.

\begin{lemma}\label{lemma:realisation}
  Let $\Gamma(x_1, \dots, x_n)$ be the same boolean formula as the one in~\cref{thm:realisation_gamma}.

  Then, there is a multigraph homomorphism $f \from G \to H$, such that
  \begin{enumerate}
  \item $p_2 \from |G^{(2)}_f| \to |\wt{G}^{(2)}_f|$ is trivial,
  \item $\Gamma_f$ is equivalent to $\Gamma$ in the sense of~\cref{thm:realisation_gamma},
  \item the edges of $H$ have two preimages,
  \item the vertices of $H$ have three preimages,
  \item the vertices of $G$ have degree 4,
  \item the vertices of $H$ have degree 6,
  \item the vertices of the graph $G^{(2)}_f$ have degree two, or, equivalently, $G^{(2)}_f$ is a disjoint
    union of circles.
  \end{enumerate}
\end{lemma}

\begin{proof}
  First note that all disjunctions that constitute $\Gamma$ split into pairs
  $(\alpha \land \beta) \to \gamma,\ (\neg \alpha \land \neg \beta) \to \neg \gamma$. After
  selecting only one disjunction from each pair, let us renumber them from 1 to $\sfrac{m}2$. Thus,
  we have $\sfrac{m}2$ disjunctions
  $(\alpha_1 \land \beta_1) \to \gamma_1, \dots, (\alpha_{\sfrac{m}2} \land \beta_{\sfrac{m}2}) \to
  \gamma_{\sfrac{m}2}$.

  We will construct the graphs $G$ and $H$ and the map $f \from G \to H$ iteratively. To begin with,
  let $G$ and $H$ be empty graphs.

  For each disjunction $(\alpha_j \land \beta_j) \to \gamma_j$, add three vertices $v^j_1$, $v^j_2$, and
  $v^j_3$ to the graph $G$, and add one vertex $v^j$ to the graph $H$, and define a map
  $f(v^j_i) = v^j$ on the new vertices $v_j^i$.

  Next, we define a set $S_i$ for each variable $x_i$, which, at the end of the procedure described
  below, will contain all the vertices of the connected component associated with the variable
  $x_i$.  For simplicity, if $y = x_i$ is a variable, denote the corresponding set $S_i$ by
  $S(y)$.

  Initially, let all $S(y)$ be empty. Then, for each of the $\sfrac{m}2$ disjunctions
  $(\alpha_j \wedge \beta_j) \to \gamma_j$, we perform the following steps:
  \begin{enumerate}
  \item If $\alpha_j = x$, add $(v^j_1, v^j_2)$ to $S(x)$. Otherwise, if $\alpha_j = \neg x$, add $(v^j_2, v^j_1)$ to $S(x)$.
  \item Similarly, if $\beta_j = y$, add $(v^j_2, v^j_3)$ to $S(y)$. If $\beta_j = \neg y$, add $(v^j_3, v^j_2)$ to $S(y)$,
  \item Finally, if $\gamma_j = z$, add $(v^j_1, v^j_3)$ to $S(z)$. If $\gamma_j = \neg z$, add $(v^j_3, v^j_1)$ to $S(z)$.
  \end{enumerate}

  After performing these steps, for each $j$ and for any $1 \leq i \neq k \leq 3$, exactly one of
  the pairs $(v^j_i, v^j_k)$ and $(v^j_k, v^j_i)$ belongs to some set $S_l$. Furthermore,
  $S_l \cap S_t = \varnothing$ for $l \neq t$, and $S_l \cap \tau(S_t) = \varnothing$ for any $l$
  and $t$. Here $\tau(S_t) = \{ (y, x)\ |\ (x, y) \in S_t \}$ is the set of ``reversed'' pairs in
  $S_t$.

  We proceed by choosing arbitrary linear orders on each set $S_l$. Using these orders, we will add edges to $G$
  and $H$ as follows.

  Let
  $S_l = \left\{ (v^{j_1}_{i_1}, v^{j_1}_{k_1}), (v^{j_2}_{i_2}, v^{j_2}_{k_2}), \dots,
    (v^{j_s}_{i_s}, v^{j_s}_{k_s}) \right\}$. For each $r = 1, \dots, s-1$ we add a new edge
  $e_{l,r}$ between $v^{j_r}$ and $v^{j_{r+1}}$ in $H$.\footnote{If there are edges between $v^{j_r}$ and
    $v^{j_{r+1}}$ in $H$, we add a new one.} We add two edges $e_{l,r}^1$ and
  $e_{l,r}^2$ connecting $v^{j_r}_{i_r}$ with $v^{j_{r+1}}_{i_{r+1}}$, and $v^{j_r}_{k_r}$ with
  $v^{j_{r+1}}_{k_{r+1}}$, respectively, to $H$, and set $f(e^1_{l,r}) = f(e^2_{l,r}) =
  e_{l,r}$.

  Additionally, we add an edge $e_{l,s}$ between $v^{j_s}$ and $v^{j_1}$ to $H$, and two
  edges $e^1_{l,s}$ and $e^2_{l,s}$ between $v^{j_s}_{i_s}$ and $v^{j_1}_{i_1}$, and between
  $v^{j_s}_{k_s}$ and $v^{j_1}_{k_1}$, respectively, to $G$. We set
  $f(e^1_{l,s}) = f(e^2_{l,s}) = e_{l,s}$.

  It is evident that at each step, two edges $((v^{j_r}_{i_r}, v^{j_r}_{k_r}), (v^{j_{r+1}}_{i_{r+1}},
  v^{j_{r+1}}_{k_{r+1}}))$ and $((v^{j_r}_{k_r}, v^{j_r}_{i_r}), (v^{j_{r+1}}_{k_{r+1}},
  v^{j_{r+1}}_{i_{r+1}}))$ are added to $G^{(2)}_f$. Thus, after processing a single set $S_l$, we
  obtain in $G^{(2)}_f$ two circles on the vertices of $S_l$ and $\tau(S_l)$. Therefore, after
  processing all the sets $S_l$, the graph $G^{(2)}_f$ consists of a disjoint union of circles on the
  vertices of the sets $S_l$ and $\tau(S_l)$.

  It is clear that $p_2 \from |G^{(2)}_f| \to |\wt{G}^{(2)}_f|$ is trivial because $S_l \cap
  \tau(S_l) = \varnothing$ for any $S_l$.

  Additionally, if we associate each $S_l$ to a variable $x_l$, it can be seen that each triple
  $\{v^j_1, v^j_2, v^j_3\}$ contributes two different disjunctions
  $(\alpha_j \land \beta_j) \to \gamma_j$ and $(\neg \alpha_j \land \neg \beta_j) \to \neg \gamma_j$
  to $\Gamma_f$. Actually, it adds 6 disjunctions, one for each permutation of the triple, but four
  of them can be derived from the other two by permutation of their terms, so we consider them as
  duplicates.

  Therefore, since these and only these disjunctions form the formula $\Gamma$, the formulas
  $\Gamma_f$ and $\Gamma$ become equivalent after elimination of the duplicates.

  If follows from the construction that the vertices of $H$ have 3 preimages, and the edges of $H$
  have 2 preimages. Furthermore, from the latter fact it follows that the map $\hat{f} \from
  \wt{G}^{(2)}_f \to H$ induced by $f \from G \to H$ is injective on the edges; thus, because each
  vertex $v$ of $H$ have 3 preimages under $\hat{f}$ (that is, unordered pairs of $f^{-1}(v)$), and
  each such preimage have degree two,\footnote{Recall that $G^{(2)}_f$ is a disjoint union of
    circles, so $\wt{G}^{(2)}_f$ is as well.} the vertices of $H$ have degree six.

  Finally, let us note that each vertex $x$ of $G$ is included into exactly 2 vertices of
  $\wt{G}_f^{(2)}$ (recall that they are unordered pairs of disctinct vertices of $G$ with the same
  image), and both of them have degree two. The construction provides that each such vertex
  corresponds to a pair of edges in $G$ incident with the vertex $x$. Therefore, the vertices of $G$
  have degree four.
\end{proof}

\begin{proof}[Proof of~\cref{thm:realisation_gamma}]
  By using~\cref{lemma:realisation} we obtain a multigraph homomorphism $f \from G \to H$ whose formula
  $\Gamma_f$ is equivalent to $\Gamma$. To construct a generic immersion $g \from S \looparrowright B$ let us
  follow the construction of $g \from S \looparrowright B$ in~\cref{thm:counterexample}.

  To see that it is applicable in this case, let us note that each vertex $v$ of $H$ has degree 6,
  and each unordered pair of its distinct preimages form a vertex in $\wt{G}^{(2)}_f$ of degree 2, hence
  it is possible to represent $v$ as an intersection of 3 squares, and assign a half-edge
  coming from the triple point of this intersection to each edge in $\wt{G}^{(2)}_f$ incident to the
  pair of preimages of $v$.
\end{proof}

It is worth noting that $B$ can be embedded in $\R^3$. Thus, a natural question arises: for which
examples is it possible to glue the boundary of $S$ to a surface embedded in $\R^3 \setminus B$ such
that we obtain a generic immersion $S \looparrowright \R^3$ that realises $\Gamma$?\footnote{Note
  that this question is different from the~\cite[Proposition 4.19]{carter_saito_book} and~\cite{li},
  where it is proven that an appropriate graph $H$ can be realised as the multiple point set of a generic
  immersion $S \looparrowright \R^3$.}

In some cases, such as the examples from~\cref{thm:counterexample} and certain examples obtained
from the construction in~\cref{ex:nontrivial_gamma} by joining the endpoints of the connected
components of $G^{(2)}_f$ to obtain circles, it can be shown that it is not possible to glue the
boundary of $S$ in such a way. One can demonstrate this by checking that the boundary of $S$ is not
null-homologous in $\R^3 \setminus H$. With a computer, checking this is straightforward, using a
sort of Wirtinger presentation for $\pi_1(\R^3 \setminus H, \Z)$ as described in~\cite[Proposition
9.1.9]{groupoids}.

\section{Approximation by embeddings}

In this section, we will establish a connection between the existence of a lifting to an embedding
and the concept of approximation by embeddings for generic maps. We will define the notion of
generic maps in our context and show that the condition of the triviality of the covering map $p_2$
becomes sufficient for the existence of a lifting in the special case of generic maps from trees to
segments.

Before we proceed, let us introduce some definitions.

\begin{definition}[approximability by embeddings]
  Let $f \from G \to H$ be a simplicial map between graphs, and let $i \from |H| \hookrightarrow S$
  be a continuous embedding of $|H|$ into a surface $S$ with a metric $d$. We say that $f$ is
  \emph{approximable by embeddings}\index{approximability by embeddings} if for every
  $\varepsilon > 0$, there exists a continuous embedding $j \from |G| \hookrightarrow S$ such that
  $d(j(x), i(|f|(x))) < \varepsilon$ for every $x \in |G|$.
\end{definition}

It is worth noting that the approximability of $f$ generally depends on the embedding $i$. However,
as we will demonstrate later, it does not depend on the choice of a metric on $S$. Moreover, as we
will show later, the approximability of $f$ is equivalent to the existence of a compatible with $f$
embedding of $|G|$ into a neighborhood of $i(|H|)$ in $S$, and, therefore, depends only on the
geometry of this neighborhood.

\begin{definition}[{ribbon graph\index{ribbon graph}, see~\cite[Definition 1.5]{graphs_surfaces} or~\cite[Definition 2.4]{minc}}]
  \emph{A ribbon graph} $\HH$ is a surface with boundary represented as the union of two sets of
  discs: a set $V(\HH)$ of vertices or 0-handles, and a set $E(\HH)$ of edges or 1-handles,
  satisfying the following conditions:
  \begin{enumerate}
  \item $D_v \cap D_w = \varnothing$ and $D_e \cap D_g = \varnothing$ for different vertices
    $v, w \in V(\HH)$ and different edges $e, g \in E(\HH)$,
  \item for each edge $D_e$ there are exactly two vertices $D_v$ and $D_w$ that intersect $D_e$;
    moreover, $D_e \cap D_v$ and $D_e \cap D_w$ are two disjoint arcs lying in the boundary of
    $D_e$, and in the boundaries of $D_v$ and $D_w$, respectively.
  \end{enumerate}
\end{definition}

By~\cite[Theorem 1.3]{hoffman_embedding}, an embedded graph $f \from |H| \to S$ lies in a
triangulation of $S$; that is, there are a homeomorphism $h_s \from S \to |K_s|$, a piecewise linear
homeomorphism $g_s \from |H| \to |H_s|$, where $H_s$ subdivides $H$, and a simplicial embedding
$f_s \from H_s \hookrightarrow K_s$, such that $f = h_s^{-1} \circ |f_s| \circ g_{s}$.\footnote{In fact,
  by applying~\cite[Theorem 1.3]{hoffman_embedding} we obtain not a simplicial complex, but a
  polyhedron with a graph embedded as a 1-subpolyhedron; however, it is easy to obtain a piecewise
  linear homeomorphism $g_s \from |H| \to |H_s|$ and a simplicial embedding $f_s \from H_s \to K_s$
  by taking suitable triangulations of these polyhedra.} Recall that the surface $S$ admits a unique
piecewise linear structure.\footnote{{See, for example,~\cite[Theorem 5, Chapter
    8]{moise}.}} Thus, $|f_s|$ is uniquely defined by $f$ up to piecewise linear
homeomorphisms of $|H_s|$ and $|K_s|$.

Therefore, given an embedding $f \from |H| \to S$, let the surface $\HH$ be a regular neighborhood
of $f(|H|)$ in $S$, decomposed into discs around the images of the vertices of $H$ (these discs form
the vertices of the ribbon graph $\HH$), and strips around the edges of $H$ (these strips,
similarly, form the edges of $\HH$).\footnote{{For the definitions of regular neighborhoods, derived
    subdivisions, and simplicial neighborhoods, the reader may refer to~\cite{rourke_sanderson}.}}

Explicitly, assume we have a simplicial embedding $f_s \from H_s \hookrightarrow K_s$. Take the
second derived subdivisions $H_s''$ and $K_s''$ of $H_s$ and $K_s$, respectively, and denote by
$f''$ the induced simplicial embedding $f'' \from H_s'' \hookrightarrow K_s''$. Denote by
$N(H_s'', K_s'')$ the simplicial neighborhood of $f''(H_s'')$ in $K_s''$, that is, the simplicial
complex consisting of the simplices that meet $f(H_s'')$, with their faces. Hence
$|N(H_s'', K_s'')|$ is a regular neighborhood of $f(H)$ in $S$. As discussed earlier, we define
$\HH$ to be $|N(H_s'', K_s'')|$.

To obtain a decomposition of $\HH$ into edges and vertices, for each vertex $v \in V(H)$, let us
consider its image under the composition
$F = |f''| \circ g \from |H| \to |H_s''| \hookrightarrow |K_s''|$, where $g$ is a homeomorphism
$|H| \cong |H_s''|$, and take the star of $F(v)$ in $N(H_s'', K_s'')$ as the disc
$D_v \in V(\HH)$.\footnote{Note that the image $F(v)$ is a vertex of $K_s''$ since $H_s''$
  subdivides $H$.}

It can be seen that these discs split $\HH$ into a set of strips, with each strip intersecting
$f(|H|)$ along exactly one curve, which is a portion of the image of one of the edges of $H$. Thus,
each edge $e$ of the graph $H$ corresponds to a strip $D_e$ containing a portion of its image. Let
us take these strips as edges of the ribbon graph $\HH$.

Since a regular neighborhood of a subpolyhedron is uniquely defined up to piecewise linear
homeomorphism, see~\cite[Theorem~3.8]{rourke_sanderson}, a ribbon graph of a graph embedded into a
surface is also uniquely defined.

Furthermore, it follows that there is a natural bijection between the vertices and the edges of a
graph $H$ embedded into a surface $S$, and the vertices and the edges of an induced ribbon graph
$\HH$. Thus, taking this into account, we will label the vertices and the edges of $H$ and $\HH$ by
the same letters; for example, a vertex $v \in V(H)$ corresponds to a vertex $D_v \in V(\HH)$, and an edge
$e = \{v, w\} \in E(H)$ corresponds to an edge $D_e = D_{\{v, w\}} \in E(\HH)$.

It is also important to note that, although our definition of a ribbon graph induced by an
embedding differs from the definition of the normal neighborhood in~\cite{minc}, the ribbon graph
$\HH$ in our construction is also a normal neighborhood in the sense of~\cite[Definition
2.4]{minc}; indeed, all conditions from the latter definition, except the second one, are
automatically satisfied. However, it is easy to see that for each edge $e$ incident to a vertex $v$
of $H$, its image $f(|e|)$ intersects the boundary of $D_v$ at exactly one point. Therefore, it
follows that the condition (2) from~\cite[Definition 2.4]{minc} is also satisfied. This allows us to
use results from~\cite{minc}, using ribbon graphs in place of normal neighborhoods.

\begin{definition}[{\cite[Definition 2.7]{minc}}]\label{def:rapprox}
  Let $f \from G \to H$ be a simplicial map between graphs, and $\HH$ be a ribbon graph
  corresponding to an embedding of $|H|$ into a surface.

  We say that $f$ is \emph{R-approximable by an embedding}\index{R-approximability} if there exists
  a continuous embedding $j \from |G| \hookrightarrow \Int \HH$ called an
  \emph{R-approximation}\index{R-approximation}, satisfying the following conditions:
  \begin{enumerate}
  \item for a vertex $v \in V(G)$, $j(v) \in D_{f(v)} \in V(\HH)$,
  \item if $f(\{v, w\})$ is a vertex of $H$, then $j(|\{v, w\}|)
    \subset D_{f(\{v, w\})}$, where
    $|\{v, w\}|$ is a segment representing the edge $\{v, w\}$ in $|G|$,
  \item if $f(\{v, w\})$ is an edge of $H$, then $j(|\{v, w\}|)$ is a simple arc intersecting
    $D_{f(v)} \cap D_{\{f(v), f(w)\}})$ at a single point $p$ and
    $D_{f(w)} \cap D_{\{f(v), f(w)\}}$ at a single point $q$.\footnote{Thus, the points
      $x = j^{-1}(p)$ and $y = j^{-1}(q)$ decompose $|\{v, w\}|$ into three subsegments
      $|\{v, x\}|$, $|\{x, y\}|$, and $|\{y, w\}|$, where $j(|\{v, x\}|) \subset D_{f(v)}$,
      $j(|\{x, y\}|) \subset D_{\{f(v), f(w)\}}$, and $j(|\{y, w\}|) \subset D_{f(w)}$.}
  \end{enumerate}

  If $j$ is a general position map $j \from |G| \to \Int \HH$ satisfying these conditions, we call
  $j$ a \emph{generic R-approximation}\index{R-approximation!generic} of $f$.
\end{definition}

\begin{lemma}[{\cite[Propositions 2.9]{minc}}]\label{lemma:Rapprox}
  Let $f \from G \to H$ be a simplicial map between graphs, and let $i \from |H| \hookrightarrow S$ be an
  embedding of $H$ into a surface $S$ with a metric. Then the following statements are equivalent:
  \begin{enumerate}
  \item $f$ is approximable by embeddings.
  \item $f$ is R-approximable by an embedding.
  \end{enumerate}
\end{lemma}

From this point onwards, we focus on stable maps between graphs. The definition provided below is
a one-dimensional simplicial version of the definition of piecewise linear stable maps
from~\cite[Appendix B]{mel2}.

\begin{definition}
  Let $f \from G \to H$ be a simplicial map between graphs.

  We call a vertex $v \in G$ \emph{regular} with respect to $f$ if $f$ is bijective on the star
  $\st(v) = \bigcup_{e \in E(G), v \in e} e$.

  The map $f$ is called \emph{stable}\index{stable map} if, for any vertex $w \in V(H)$,
  \begin{enumerate}
  \item if $\deg(w) \neq 2$, all the vertices in $f^{-1}(w)$ are regular,
  \item if $\deg(w) = 2$, $f^{-1}(w)$ contains at most one non-regular vertex.
  \end{enumerate}
\end{definition}

If a simplicial map $f$ is stable, it must also be non-degenerate. Indeed, if $f(e) = v$, where $e$
is an edge and $v$ is a vertex, the endpoints of $e$ are non-regular with respect to $f$, implying
that $f$ is not stable.

Now we are ready to state the theorem.

\begin{theorem}\label{thm:approx}
  Let $f \from G \to J$ be a stable simplicial map from a graph $G$ to a path graph $J$ (that is, a
  triangulation of a segment). Let $i \from |J| \hookrightarrow \R^2$ be an embedding. Then the
  following statements are equivalent:
  \begin{enumerate}
  \item $f$ is approximable by embeddings with respect to $i \from |J| \to \R^2$.
  \item $|f|$ lifts to an embedding.
  \end{enumerate}
\end{theorem}

Let us start by proving a simple auxiliary lemma.

\begin{lemma}\label{lemma:square}
  Let $S$ be a rectangular region $[C, D] \times [H, L]$ in $\mathbb{R}^2$. Let $\alpha$ and $\beta$ denote the sides
  $\{C\} \times [H, L]$ and $\{D\} \times [H, L]$, respectively. Let $\gamma$ be a simple curve
  connecting a point $a \in \alpha$ and a point $b \in \beta$, such that
  $\gamma \cap (\alpha \sqcup \beta) = \{a, b\}$. Assume $\Phi$ is a path-connected subset of $S$
  not intersecting $\gamma$, such that $\Phi \cap \alpha$ and $\Phi \cap \beta$ are finite sets of
  points.

  Then, for all $p \in (\Phi \cap \alpha) \sqcup (\Phi \cap \beta)$, either $p_y < s(p)_y$ or
  $p_y > s(p)_y$ holds simultaneously. Here $s(p) = a$ if $p \in \alpha$, $s(p) = b$ if $p \in \beta$, and
  $q_y$ denotes the $y$-coordinate of a point $q$.
\end{lemma}
\begin{proof}
  Suppose there exists a pair of points
  $p, q \in (\Phi \cap \alpha) \sqcup (\Phi \cap \beta)$ such that $p_y < s(p)_y$ and
  $q_y > s(q)_y$, or $p_y > s(p)_y$ and $q_y < s(q)_y$. Since $\Phi$ is path-connected and does not
  intersect $\gamma$, there must be a path $\phi$ in $\Phi$ (and thus in $S \setminus \gamma$)
  connecting $p$ and $q$. However, according to \cite[Theorem 8, Chapter 2]{moise},
  $S \setminus \gamma$ consists of two path-connected components with $p$ and $q$ lying in different
  components. This contradicts the existence of the path $\phi$.
\end{proof}

\begin{proof}[Proof of~\cref{thm:approx}]
  First, let us number the vertices of $J$ according to an arbitrary orientation on $J$ and denote
  them by $v_0, \dots, v_N$.

  Observe that any embedding $i \from |J| \hookrightarrow \mathbb{R}^2$ induces the same ribbon
  graph $\JJ$ up to a piecewise linear homeomorphism.\footnote{Given two distinct ribbon graphs
    $\JJ$ and $\JJ'$, we can take arbitrary homeomorphisms between arcs $D_v \cap D_e$ and
    $D'_v \cap D'_e$ and then extend them to a homeomorphism $\JJ \cong \JJ'$ that preserves the
    handle decompositions of $\JJ$ and $\JJ'$.} Thus, it follows from~\cref{lemma:Rapprox} that the approximability of
  $f$ by embeddings does not depend on the embedding $i$.

  Therefore, for convenience, we can assume that $i$ embeds $|J|$ into $\R^2$ with the Euclidean
  distance $d_2$ as a line segment with $i(v_k) = \left(2k+\frac12, 0\right)$, and $\JJ$ is a
  ribbon $[0, 2N+1] \times [-1, 1]$ in $\R^2$, where $D_{v_k} = [2k, 2k+1] \times [-1, 1]$ and
  $D_{\{v_k, v_{k+1}\}} = [2k+1, 2(k+1)] \times [-1, 1]$.\footnote{$\JJ$ collapses onto $i(|J|)$,
    hence it is a regular neighborhood by~\cite[Corollary 3.30]{rourke_sanderson}.}

  Now suppose that $|f|$ lifts to an embedding $\wt{|f|} \from |G| \to |J| \times \R$. Let $M$ be
  the maximum of $|h(p)|$, where $h(p) = \pr_\R \wt{|f|}(p)$. For any $\varepsilon > 0$, let
  $F \from |G| \to \R^2$ be the composition $S_\varepsilon \circ \wt{|f|}$ where
  $S_\varepsilon(q, y) = \left(i(q)_x, \frac{\varepsilon y}{2M}\right)$; here $i(q)_x$
  stands for the $x$-coordinate of $i(q)$. Therefore, putting $q = f(p)$, we have
  \[
    d_2\left(F(p), i(q)\right) = d_2\left( \left(i(q)_x, \frac{\varepsilon h(p)}{2M}\right), \left(
        i(q)_x, 0 \right) \right) = \left| \frac{\varepsilon h(p)}{2M} \right| < \varepsilon.
  \]
  Thus, $f$ is approximable by embeddings.

  Now let us prove the reverse implication. Suppose $f$ is approximable by
  embeddings. By~\cref{lemma:Rapprox}, there exists an R-approximation
  $j \from |G| \hookrightarrow \Int \JJ$. We are going to define orders on the sets $f^{-1}(v_k)$
  and $f^{-1}(\{v_k, v_{k+1}\})$. It will be more convenient to start with the sets
  $f^{-1}(\{v_k, v_{k+1}\})$.

  Consider the images of the edges from $f^{-1}(\{v_k, v_{k+1}\})$ in the rectangular regions
  $D_{\{v_k, v_{k+1}\}}$. It follows from the definition of R-approximation that
  $j(|e|) \cap D_{\{v_k, v_{k+1}\}}$ for an edge $e \in f^{-1}(\{v_k, v_{k+1}\})$ is a simple curve
  $\gamma_e$ going from the left side of $D_{\{v_k, v_{k+1}\}}$ to its right side, and intersecting
  the boundary of $D_{\{v_k, v_{k+1}\}}$ only at the endpoints. Hence, by
  applying~\cref{lemma:square}, for each pair of edges $e, g \in f^{-1}(\{v_k, v_{k+1}\})$, the
  endpoints of $\gamma_e$ have both either larger or smaller $y$-coordinates than the endpoints of
  $\gamma_g$ lying in the corresponding sides of the rectangle. Therefore, we can define a linear
  order by putting $e \prec_{k, k+1} g$ if both endpoints of $\gamma_e$ have smaller $y$-coordinates
  than the corresponding endpoints of $\gamma_g$.

  Now let us proceed to the sets $f^{-1}(v_k)$. The situation with them, although slightly more
  complex, is essentially the same. In what follows, we put $\gamma_w = j(|\st w|) \cap D_{f(w)}$
  for a vertex $w \in V(G)$.

  First note that, since $f$ is stable, all the vertices in $f^{-1}(v_0)$ and $f^{-1}(v_N)$ have
  degree one, and for any $w \in f^{-1}(v_0)$ the set $\gamma_w$ is a simple curve connecting $j(w)$
  with a point on the right side of $\partial D_{v_0}$, and this point is the only intersection of
  $\gamma_w$ with the boundary of $D_{v_0}$. Therefore, we can order the vertices $\{ w_i \}$ in
  $f^{-1}(v_0)$ with respect to the order of the $y$-coordinates of the points
  $\gamma_{w_i} \cap \partial D_{v_0}$, which gives us a linear order $\prec_0$ on $f^{-1}(v_0)$.

  The same holds true for vertices from the set $f^{-1}(v_N)$, with the only difference being that
  the curves $\gamma_w$ for $w \in f^{-1}(v_N)$ connect $j(w)$'s with points on the left side of
  $D_{v_N}$. As for $f^{-1}(v_0)$, we can order the vertices in $f^{-1}(v_N)$ according to the order
  of the $y$-coordinates of the points $\gamma_w \cap \partial D_{v_N}$.

  Now consider the set $f^{-1}(v_k)$, where $0 < k < N$. Since $f$ is stable, for any pair of
  vertices $w, s \in f^{-1}(v_k)$, one of the sets $\gamma_w$ and $\gamma_s$ is a simple curve
  connecting a point on the left side of $D_{v_k}$ and a point on its right side, and intersecting
  the boundary of $D_{v_k}$ only at these two points. At the same time, the other set is clearly
  path-connected, and it intersects the boundary of $D_{v_k}$ in a finite number of points, these
  points will be referred to as ``endpoints''. Therefore, by applying~\cref{lemma:square} we get
  that the $y$-coordinates of the endpoints of $\gamma_w$ are all larger or all smaller than the
  $y$-coordinates of the endpoints of $\gamma_s$ once we compare only the pairs of points lying on
  the same side of $D_{v_k}$. Hence we can define a linear order $\prec_k$ by letting $w \prec_k s$
  if all the endpoints of $\gamma_w$ lying on the left and right sides of $D_{v_k}$ have the smaller
  $y$-coordinates than all the endpoint of $\gamma_s$ lying on the same sides.

  Finally, we are going to prove that $\{ \prec_k,\ 0 \leq k \leq N \}$ forms an admissible
  collection of linear orders, and therefore, $|f|$ lifts to an embedding according
  to~\cref{thm:orders}. Indeed, take a pair of edges $e = \{w_e, s_e\}, g = \{ w_g, s_g \}$ of $G$
  such that $f(w_e) = f(w_g) = v_k$ and $f(s_e) = f(s_g) = v_{k+1}$. Without loss of generality,
  assume $e \prec_{k, k+1} g$. This implies that the ``left'' endpoint of $\gamma_e$, lying on the
  left side of $D_{\{v_k, v_{k+1}\}}$, has a smaller $y$-coordinate than the ``left'' endpoint of
  $\gamma_g$. Hovewer, the ``left'' endpoints of $\gamma_e$ and $\gamma_g$ are endpoints of
  $\gamma_{w_e}$ and $\gamma_{w_g}$, lying on the same side of $D_{v_k}$; hence, $w_e \preceq_{k}
  w_g$. Repeating this reasoning for the ``right'' endpoints of $\gamma_e$ and $\gamma_g$, we
  conclude that $s_e \preceq_{k+1} s_g$. Therefore, the orders $\prec_k$ form an admissible
  collection of linear orders.
\end{proof}

The next example shows that the assumption that $f$ is stable in~\cref{thm:approx} is needed.

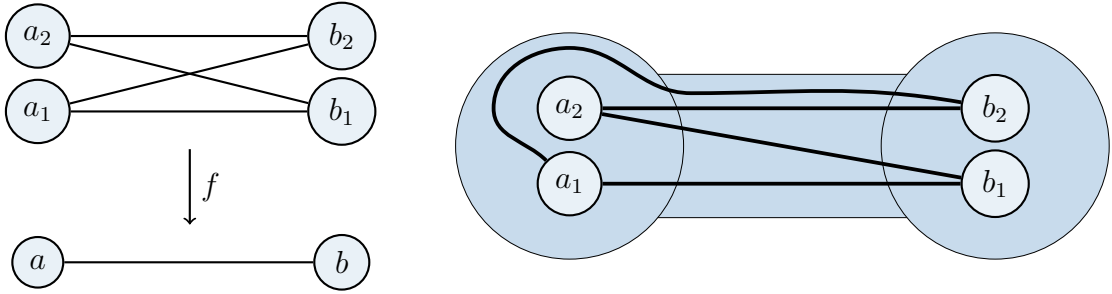
\begin{figure}[htb!]
  \centering
  \begin{subfigure}{.4\textwidth}
    \centering
    \begin{tikzpicture}[thick, main node/.style={circle,fill=blue!10,draw,font=\sffamily\bfseries} ]
      \node[main node] at (-2, 0) (a1) {$a_1$};
      \node[main node] at (-2, 1) (a2) {$a_2$};
      \node[main node] at (2, 0) (b1) {$b_1$};
      \node[main node] at (2, 1) (b2) {$b_2$};

      \node[main node] at (-2, -2) (a) {$a$};
      \node[main node] at (2, -2) (b) {$b$};

      \draw [thick, ->] (0, -0.5) -- (0, -1.5) node[midway,right] {$f$};

      \draw (a1) to (b1);
      \draw (b1) to (a2);
      \draw (a2) to (b2);
      \draw (b2) to (a1);
      \draw (a) to (b);
    \end{tikzpicture}
  \end{subfigure}
  \begin{subfigure}{.58\textwidth}
    \centering
    \begin{tikzpicture}[thick, main node/.style={circle,fill=blue!10,draw,font=\sffamily\bfseries}]
      \node[rectangle,fill=color0!25,minimum width=7cm,minimum height=1.9cm,draw,line width=0.1mm] (r) {};

      \node[circle,fill=color0!25,minimum size=3cm,draw,line width=0.1mm] at ([shift=(180:2.8cm)]r) (a) {};
      \node[main node] at ([shift={(90:-0.5cm)}]a) (a1) {$a_1$};
      \node[main node] at ([shift={(90:0.5cm)}]a) (a2) {$a_2$};

      \node[circle,fill=color0!25,minimum size=3cm,draw,line width=0.1mm] at ([shift={(0:2.8cm)}]r) (b) {};
      \node[main node] at ([shift={(90:-0.5cm)}]b) (b1) {$b_1$};
      \node[main node] at ([shift={(90:0.5cm)}]b) (b2) {$b_2$};

      \draw[line width=0.5mm] (a1) to (b1);
      \draw[line width=0.5mm] (b1) to (a2);
      \draw[line width=0.5mm] (a2) to (b2);
      \draw[line width=0.5mm] (b2) to[out=170,in=0] ([xshift=-4cm,yshift=0.2cm]b2) to[out=180,in=0] ([yshift=0.8cm]a2)
      to[out=180,in=90] ([xshift=-1cm]a2) to[out=-90] (a1);
    \end{tikzpicture}
    \vspace{1em}
  \end{subfigure}
  \caption{$f \from G \to I \hookrightarrow \R^2$ and its R-approximation.}\label{fig:notstable}
\end{figure}

\begin{example}
  Let $f$ be a map $G \to J$, where $G$ triangulates the circle $S^1$, that is shown on the
  left-hand side of~\cref{fig:notstable}. It is R-approximable by embeddings, see the right-hand
  side of~\cref{fig:notstable}, therefore, it is approximable by embeddings. However, it does not
  lift to an embedding, since there is a 2-obstructor for $f$ defined as the path
  $(a_1, a_2) \to (b_1, b_2) \to (a_2, a_1)$ in $G^{(2)}_f$.
\end{example}

Now we use some known results on the problem of existence of approximation by embeddings to
obtain sufficient conditions for the existence of a lifting to an embedding in the case of stable
maps from trees to segments.

\begin{definition}
  Let $f \from G \to H$ be a simplicial map between graphs, let $i \from |H| \hookrightarrow S$ be
  an embedding of $H$ into a surface $S$, and let $\HH$ be a corresponding ribbon graph.  We call a
  generic R-approximation $j \from |G| \to \Int \HH$ a
  \emph{$\Z_2$-approximation}\index{$Z_2$-approximation} if for any pair of non-adjacent edges
  $e, g \in E(G)$ (i.e., $e \cap g = \varnothing$) the set $j(|e|) \cap j(|g|)$ is finite and
  consists of an even number of points.

  $f$ is called \emph{$\Z_2$-approximable}\index{$Z_2$-approximability} if there exists a $\Z_2$-approximation of $f$.
\end{definition}

The following theorem is proved in~\cite{fulek} (in general case) and in~\cite{skop} (for the case when $T$ is a
trivalent tree):
\begin{theorem}[{\cite[Corollary 3]{fulek} and~\cite[Theorem 1.5]{skop}}]\label{thm:z2_imp_app}
  Let $f \from T \to J$ be a simplicial map from a tree $T$ to a path graph $J$, and assume that $J$ is
  embedded into $\R^2$. If $f$ is $\Z_2$-approximable then $f$ is approximable by embeddings.
\end{theorem}

The variant of van Kampen obstruction for $\Z_2$-approximability of $f$ by embeddings described
below is taken from~\cite{skop} and~\cite{repovs_skop}.

Assume we have a simplicial map $f \from G \to H$ between graphs $G$ and $H$, and $H$ is embedded
into a surface $S$. Let $\Gc$ be the two-dimensional cubical complex
$\bigcup \{ e \times g\ |\ e, g \in E(G) \colon e \cap g = \varnothing \}$. Clearly, the
transposition $\tau \in S_2$ acts on $|\Gc|$ by $\tau(p \times q) = q \times p$. Let
$\wGc$ be a two-dimensional cell complex $\Gc / S_2$, whose cells are images of the
cubes of $\Gc$ under the quotient map. Therefore, 2-cells of $\wGc$ might be associated
with unordered pairs of disjoint edges of $G$.

Take any generic R-approximation
$j \from |G| \to \HH$ of $f$, where $\HH$ stands for a ribbon graph corresponding to the embedding
$|H| \hookrightarrow S$; by ``generic'' here we mean that $j$ is injective on $|V(G)|$, and for any
pair of edges $e, g \in E(G)$ the set $j(|e|) \cap j(|g|)$ consists of a finite number of points.

Let $\Gcf \subset \Gc$ be the subcomplex
$\bigcup \{ e \times g\ |\ e, g \in E(G) \colon e \cap g = \varnothing, f(e) \cap f(g) = \varnothing
\}$. Since $|\Gcf|$ is closed under the involution on $|\Gc|$, the induced involution on
$|\Gcf|$ is well-defined. Let $\wGcf$ be $\Gcf / S_2$.

Let us define a cellular 2-cochain $c_v \in C^2(\wGc; \Z_2)$ by putting
$c_v(e \cdot g) = |j(|e|) \cap j(|g|)| \bmod{2}$, where $e \cdot g$ is the image of $e \times g$
under the projection $\Gc \to \wGc$. Observe that $c_v$ is in fact a relative cochain lying in
$C^2(\wGc, \wGcf; \Z_2)$. Indeed, for edges $e, g \in E(G)$ with
$f(e) \cap f(g) = \varnothing$, the curves $j(|e|)$ and $j(|g|)$ lie in different discs with respect
to the chosen generic R-approximation and therefore do not intersect. Let
$v_f = [c_v] \in H^2(|\wGc|, |\wGcf|; \Z_2)$ be the corresponding cohomology class.

It can be shown that $v_f$ is well-defined by following~\cite[Lemmas 3.2--3.5]{shapiro}. These lemmas
consider the case of the classical van Kampen obstruction, but it is easy to adapt them to our
case.\footnote{Note that $v_{f}$ coincides with the classical van Kampen obstruction for graphs when $f
  \from G \to \{p\} \hookrightarrow \R^2$.}

\begin{theorem}[{\cite[Proposition 3.1]{skop}}]\label{thm:z2_approx}
  Suppose $f \from G \to H$ is a simplicial map between graphs $G$ and $H$, and $H$ is embedded into an
  oriented surface $S$. Then $v_f = 0$ if and only if $f$ is $\Z_2$-approximable.
\end{theorem}

Now let us introduce a van Kampen-like obstruction for the existence of a lifting. We start with a
cochain-free definition of the corresponding cohomology class inspired by~\cite{bestvina}
and~\cite{mel_van_kampen}. Next, we present an explicit cochain representing it, defined in a way
similar to the cochain $c_{v}$ defining $v_f$.

Recall that $f$ induces the principal $S_2$-bundle $p_2 \from |G^{(2)}_f| \to |\wt{G}^{(2)}_f|$. Let
$g \from |\wt{G}^{(2)}_f| \to \RP^{\infty}$ be a map classifying it.\footnote{The reader may refer
  to~\cite{principal_bundles},~\cite{milnor_stasheff}, or~\cite{hatcher_bundles} for the needed information on
  universal bundles and classifying maps.} Let $w_1$ be the generator of
$H^1(\RP^{\infty}; \Z_2) \cong \Z_2$. We define $w_f \in H^1(|\wt{G}^{(2)}_f|, \Z_2)$ as the pullback
$g^* w_1$.\footnote{{See~\cite[Theorem 7.1]{milnor_stasheff} for a computation of the ring
  $H^\bullet(\RP^\infty; Z_2)$.}}

It follows from the definition that $w_f$ is the first Stiefel–Whitney class of the one-dimensional
real vector bundle associated with the covering $p_2$.\footnote{See~\cite[Theorem
  3.1]{hatcher_bundles} or~\cite[Theorem 7.1]{milnor_stasheff}.} Indeed, this vector bundle is a
pullback under $g$ of the canonical line bundle over $\RP^\infty$, and the latter bundle is
associated with the universal bundle $S^\infty \to \RP^\infty$. Therefore, $w_f = 0$ if and only if
the bundle associated with $p_2$ is orientable, that is, admits a non-vanishing section. Hence
$w_f = 0$ if and only if $p_2$ is trivial.

Now let $\wt{f} \from |G| \to |H| \times \R$ be a generic lifting of $f$; as before, the word
``generic'' means that $\wt{f}$ is injective on $V(G)$, and for any pair of edges $e, g \in E(G)$,
the set $\wt{f}(|e|) \cap \wt{f}(|g|)$ consists of a finite number of points. Let $c_w$ be a cochain
in $C^1(\wt{G}^{(2)}_f; \Z_2)$ that takes the value $|\wt{f}(|e|) \cap \wt{f}(|g|)| \bmod{2}$ on an edge $e
\cdot g \in E(\wt{G}^{(2)}_f)$, where by $e \cdot g$ we denote the edge of $\wt{G}^{(2)}_f$
corresponding to an unordered pair of edges $e$ and $g$.

\begin{lemma*}
  $w_f = [c_w]$.
\end{lemma*}

\begin{proof}
  Let $F \from |G| \to |H| \times \R \times \R$ be an embedding such that
  $\wt{f} = \pr_{|H| \times \R} \circ F$, and let $\rho \from |\wt{G}^{(2)}_f| \to \RP^1$ be a map
  that takes an unordered pair $\{p, q\}$ of points of $|G|$ to an equivalence class of lines
  parallel to a vector going from the point $\pr_{\R \times \R}(F(q))$ to the point
  $\pr_{\R \times \R}(F(p))$.

  Clearly, we can choose $F$ so that $\rho$ is a local homeomorphism at each
  $\{p, q\} \in |\wt{G}^{(2)}_f|$ where $\wt{f}(p) = \wt{f}(q)$.

  It can be seen that $p_2$ is the pullback of the universal $S_2$-bundle under the composition
  $i \circ \rho$, where $i$ is the inclusion $\RP^1 \hookrightarrow \RP^\infty$; indeed, it follows
  from the universal property of pullbacks that there is a morphism between principal $S_2$-bundles
  $(i \circ \rho)^* (S^\infty \to \RP^\infty)$ and $p_2$, which must be an isomorphism by~\cite[Proposition
  2.1]{principal_bundles}.

  As the embedding $i \from \RP^1 \hookrightarrow \RP^\infty$ induces an isomorphism
  $H^1(\RP^\infty; \Z_2) \cong H^1(\RP^1; \Z_2)$,\footnote{The induced map of (cellular) cochain
    complexes is essentially truncating the cochains of dimension two and higher. Since all
    differentials in both complexes are zero, we obtain the desired isomorphism.}
  $w_f = \rho^* w_1$, where $w_1$ is the generator of $H^1(\RP^1; \Z_2)$.

  Note that
  $\rho^{-1}([0 \times \R]) = \left\{\{p, q\} \mid \wt{f}(p) = \wt{f}(q)\right\} \subset
  |\wt{G}^{(2)}_f|$.  Let us consider a (simplicial) cycle $C$. Clearly, there is an Euler cycle
  going through the edges belonging to $C$.\footnote{Recall that an Euler cycle is a cycle that
    visits every edge exactly once. In our case, it visits each edge in $C$ exactly once and does
    not include other edges.} Let $\psi_{C}$ be a non-degenerate piecewise linear map
  $\psi_{C} \from S^{1} \to |\wt{G}^{(2)}_f|$ realising it. According to the definition of $c_{w}$,
  \[
    c_{w}(C) = \left|\psi_{C}^{-1}\left(\left\{\{p, q\} \mid \wt{f}(p) = \wt{f}(q)\right\}\right) \right| \bmod{2} = \left| (\rho
      \circ \psi_{C})^{-1}(\lbrack 0 \times \R \rbrack) \right| \bmod{2}.
  \]
  Since $\rho$ is a local homeomorphism at the points $\{p, q\}$ where $\wt{f}(p) = \wt{f}(q)$, and
  $\psi_{C}$ is a local homeomorphism at the preimages of the interior points of the edges of
  $\wt{G}^{(2)}_f$, the last number is, in fact, the sum of the local degrees modulo 2 at the
  preimages of $[0 \times \R]$. Therefore, $[c_{w}]([C])$ is the degree modulo 2 of the map
  $\rho \circ \psi_{C}$.

  On the other hand, $w_{f}(\lbrack C \rbrack) = w_{1}(\lbrack \rho \circ \psi_{C}(S^{1}) \rbrack)$
  is also equal to the degree modulo 2 of the map $\rho \circ \psi_{C}$. Thus,
  $w_{f} = \lbrack c_{w} \rbrack$.\footnote{Note that
    $H^{1}(X;\Z_{2}) \cong \mathrm{Hom}_{\Z_{2}}(H_{1}(X;\Z_{2}), \Z_{2})$ by the universal
    coefficient theorem.}
\end{proof}

Now, we will prove the following theorem:

\begin{theorem}\label{thm:cohomology_equivalence}
  Let $f \from G \to J$ be a stable simplicial map from a graph to a path graph. Then $w_f = 0$ if
  and only if $v_f = 0$.
\end{theorem}

\begin{proof}[Sketch of a geometrical proof]
  Take a small neighborhood of $|\wt{G}^{(2)}_f|$ in $|\wGc|$ homeomorphic to a one-dimensional disc
  bundle over $|\wt{G}^{(2)}_f|$. After taking suitable triangulations, we can observe that the
  total space of the bundle is the closure of the complement of $|\wGcf|$ in $|\wGc|$
  (see~\cref{fig:book}). Therefore, the cohomology class $v_f$ belongs to the second cohomology
  group of the Thom space of this bundle. Furthermore, it can be shown that the Thom isomorphism
  maps $w_f$ to $v_f$.
\end{proof}

\begin{proof}[Cochain-level proof]
  Let $v_0, \dots, v_N$ be the vertices of $J$. By replacing the triangulations of $G$ and $J$ with
  their subdivisions, we can assume that if $f^{-1}(v_i)$ contains a non-regular vertex, then all the
  preimages of $v_{i-1}$ and $v_{i+1}$ are regular.

  Let $\wt{|f|} \from |G| \to |J| \times \R$ be a generic lifting of $f$. By choosing
  $S_\varepsilon$ as in the proof of~\cref{thm:approx}, we obtain a generic approximation of
  $f$.

  Note that the cellular cochains in $C^\bullet(\wGc, \wGcf; \Z_2)$ and the simplicial cochains in
  $C^\bullet(\wt{G}^{(2)}_f; \Z_2)$ could be interpreted as, respectively, cubical cochains in
  $C^\bullet(\Gc, \Gcf; \Z_2)$ and simplicial cochains in $C^\bullet(G^{(2)}_f; \Z_2)$, equivariant
  (as functions) with respect to the involution. Thus, let $c_v \in C^2(\Gc, \Gcf; \Z_2)$ be the
  equivariant cubical 2-cochain defining $v_f$ and corresponding to the generic approximation, and
  let $c_w \in C^\bullet(G^{(2)}_f; \Z_2)$ be the equivariant simplicial 1-cochain defining $w_f$
  and corresponding to $\wt{f}$.

  Note that $G^{(2)}_f$ can be naturally embedded into $\Gc$ as the diagonals of the squares
  $(a, b) \times (c, d)$ where $f(a) = f(c)$ and $f(b) = f(d)$. It follows from the definition of
  $c_w$ and $c_v$ that $c_v$ takes the value 1 on exactly those squares that contain a diagonal
  lying in $G^{(2)}_f$ on which $c_w$ takes 1. The opposite is also true: $c_w$ takes the value 1 on
  the diagonals of those squares on which $c_v$ takes the value 1.

  It can be observed that $|G^{(2)}_f|$ lies in the complement of $|\Gcf|$ in $|\Gc|$.
  Furthermore, as a subcomplex of $\Gc$, $\Gcf$ precisely consists of those squares
  and edges that do not intersect $|G^{(2)}_f|$.

  Consider two regular vertices $a, b \in V(G)$ with the same image $f(a) = f(b) = v_i$ lying in the
  interior of $J$ (i.e.,~$0 < i < N$). Let $a$ be incident to edges $e_1$ and $e_2$, and $b$
  be incident to edges $g_1$ and $h_1$, where $f(g_1) = f(e_1)$ and $f(h_1) = f(e_2)$. The star of
  $a \times b$ in $\Gc$ consists of a large square divided into four smaller squares
  $\{e_1, e_2\} \times \{g_1, h_1\}$.\footnote{The star of a vertex
    $v$ in a cubical complex $C$ is the union of cubes of $C$ that have $v$ as a vertex.} Observe
  also that $G^{(2)}_f$ forms a diagonal of the larger square composed of diagonals of
  $e_1 \times g_1$ and $e_2 \times h_1$.

  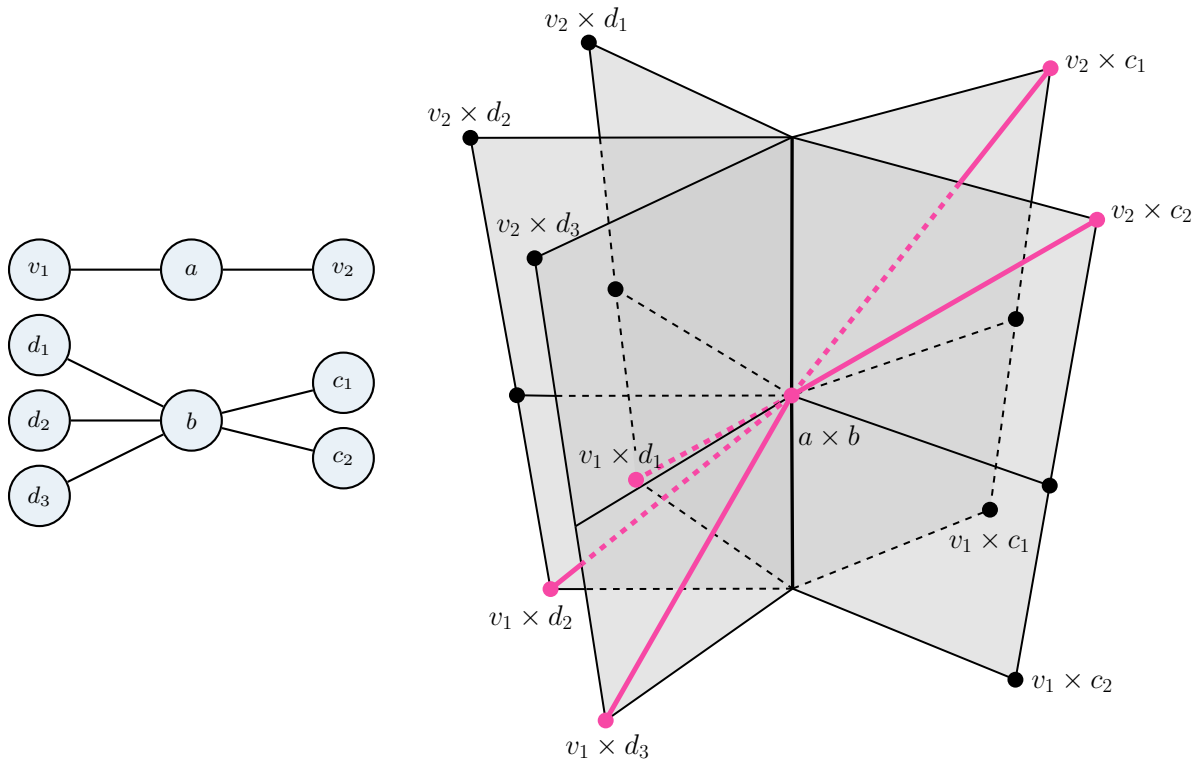
\begin{figure}[h!]
    \centering
    \raisebox{-0.5\height}{
    \begin{tikzpicture}[thick, main
      node/.style={circle,fill=blue!10,draw,minimum size=0.8cm,font=\sffamily\scriptsize\bfseries}]
      \node[main node] (b) at (0, 0) {$b$};
      \node[main node] (a) at (0, 2) {$a$};

      \node[main node] (v2) at ([xshift=2cm]a) {$v_2$};
      \node[main node] (v1) at ([xshift=-2cm]a) {$v_1$};

      \node[main node] (c1) at ([xshift=2cm,yshift=0.5cm]b) {$c_1$};
      \node[main node] (c2) at ([xshift=2cm,yshift=-0.5cm]b) {$c_2$};

      \node[main node] (d1) at ([xshift=-2cm,yshift=1cm]b) {$d_1$};
      \node[main node] (d2) at ([xshift=-2cm]b) {$d_2$};
      \node[main node] (d3) at ([xshift=-2cm,yshift=-1cm]b) {$d_3$};

      \draw (a) -- (v1);
      \draw (a) -- (v2);
      \draw (b) -- (d1);
      \draw (b) -- (d2);
      \draw (b) -- (d3);
      \draw (b) -- (c1);
      \draw (b) -- (c2);
    \end{tikzpicture}}
    \hspace{1em}
    \raisebox{-0.5\height}{\includegraphics[width=.65\linewidth,height=10cm,keepaspectratio]{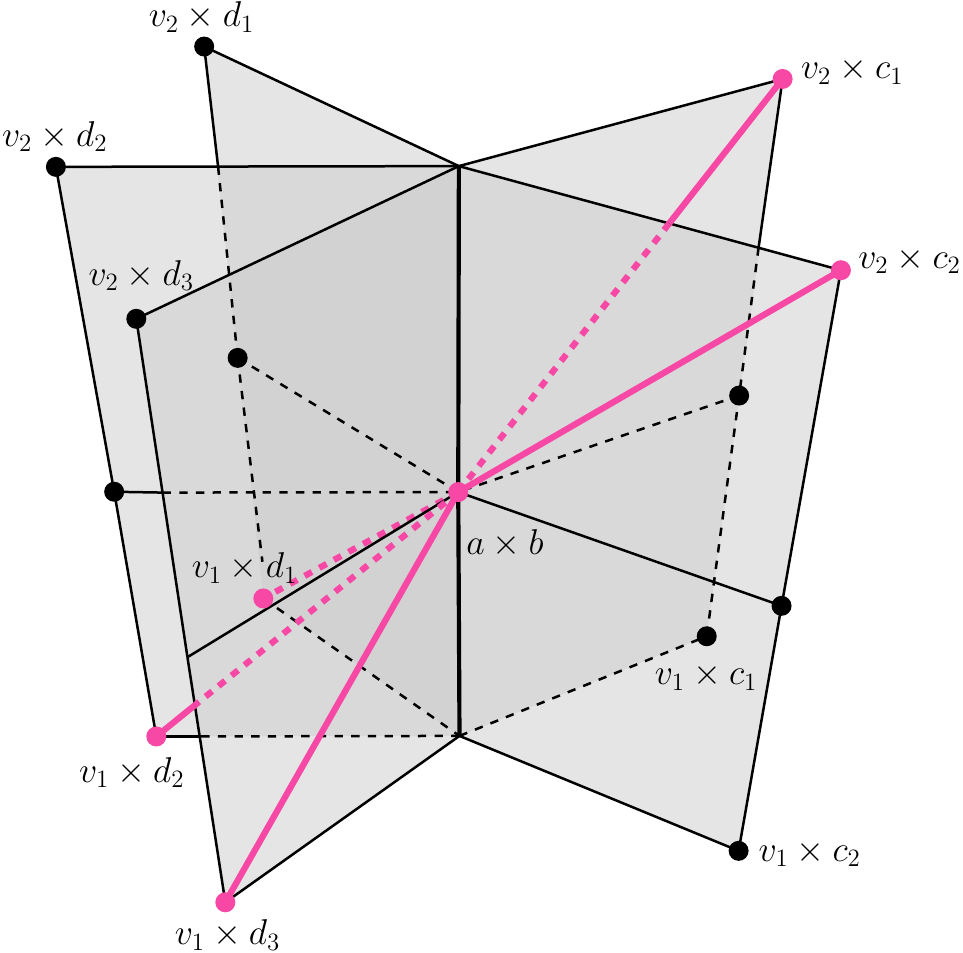}}
    \caption{The neighborhood of $a \times b$ in $\protect\Gc$. The graph $G^{(2)}_f$ is represented by
      the pink lines.}\label{fig:book}
  \end{figure}

  Now, let us consider the case where $b$ is a non-regular vertex (the case where $a$ is non-regular
  is symmetrical). We can assume that $b$ is incident to edges $g_1, \dots, g_k, h_1, \dots, h_l$,
  and $a$ is incident to edges $e_1$ and $e_2$ such that $f(g_1) = f(g_2) = \dots = f(g_k) = f(e_1)$
  and $f(h_1) = f(h_2) = \dots = f(h_l) = f(e_2)$. As illustrated in~\cref{fig:book}, the star of
  $a \times b$ in $\Gc$ is a $(k+l)$-page book, with the spine consisting of two segments, namely
  $e_1 \times b$ and $e_2 \times b$. Each page corresponds to an edge
  $v \in \{g_1, \dots, g_k, h_1, \dots, h_l\}$ that is incident to $b$ and consists of two
  half-pages, namely $e_1 \times v$ and $e_2 \times v$. Regarding $G^{(2)}_f$ around $a \times b$,
  it consists of $k+l$ diagonals of half-pages, one in each page. Moreover, in the pages
  $\{e_1, e_2\} \times g_i$, those are diagonals of the half-pages $e_1 \times g_i$, while in
  the pages $\{e_1, e_2\} \times h_i$, those are diagonals of the half-pages $e_2 \times h_i$.

  If $a, b \in f^{-1}(v_0)$ or $a, b \in f^{-1}(v_N)$, both $a$ and $b$ are regular degree one
  vertices. Therefore, the star of $a \times b$ in $\Gc$ is a single square formed by
  the edges incident to them, and $G^{(2)}_f$ is a diagonal of that square.

  Hence, the closure of the complement of $\Gcf$ in $\Gc$ can be seen as a collection of books
  connected with strips running along the graph $G^{(2)}_f$. Note that $\Gc$ does not contain any
  edges where both endpoints belong to $G^{(2)}_f$.

  For each $a \times b \in V(G^{(2)}_f)$, we choose a subset $E(a \times b)$ of edges incident with
  $a \times b$ in $\Gc$ as follows. When $a \times b$ is formed by a pair of regular degree one
  vertices, we simply select one of the two edges incident with $a \times b$, no matter which one,
  and put it in $E(a \times b)$. In the case where $a$ is a regular degree two vertex, and $b$ is a
  regular or non-regular vertex, we define $E(a \times b)$ as either
  $\{a \times h_1, \dots, a \times h_l, e_1 \times b \}$ or
  $\{a \times g_1, \dots, a \times g_k, e_2 \times b \}$, using the notation provided earlier. The
  selection of $E(b \times a)$ is done symmetrically with respect to the involution, meaning that
  $E(b \times a) = \tau(E(a \times b)) = \{ y \times x\ |\ x \times y \in E(a \times b) \}$.

  We refer to the squares of $\Gc$ that contain the edges of $G^{(2)}_f$ (on the diagonal) as
  full squares, while the squares that do not contain such edges are referred to as empty
  squares.

  An important observation is that for any vertex $a \times b \in V(G^{(2)}_f)$, the following holds
  for the squares of its star:
  \begin{enumerate}
  \item for a full square, exactly one of its sides lies in $E(a \times b)$,
  \item for an empty square, either zero or two of its sides lie in $E(a \times b)$.
  \end{enumerate}
  This property is evident for pairs of regular degree one vertices. For pairs of other types, it
  can be verified through direct calculations, considering the explicit description provided earlier
  of how $G^{(2)}_f$ intersects the star of $a \times b$.

  Suppose $w_f = 0$, and, therefore, there exists an equivariant 0-cocycle
  $c'_w \in C^0(G^{(2)}_f; \Z_2)$ with $\delta c'_w = c_w$. Let us introduce an equivariant
  1-cocycle $c'_v \in C^1(\Gc, \Gcf; \Z_2)$.

  For any edge $x \times y$ incident with $a \times b$ and belonging to $E(a \times b)$, we define
  $c'_v(x \times y) = c'_w(a \times b)$; for other edges, we set $c'_v(x \times y) = 0$. Since the
  sets $E(a \times b)$ are compatible with the involution, and given that $c'_w$ is equivariant, it
  follows that $c'_v$ is equivariant as well. Additionally, note that $c'_v$ assigns zero values to
  edges not intersecting $G^{(2)}_f$. Therefore, $c'_v$ is well-defined.

  Based on the properties of the sets $E(a \times b)$ discussed above, we observe the following:
  \begin{enumerate}
  \item On empty squares, the coboundary $\delta c'_v$ evaluates to zero. This is because each empty
    square contains an even number of edges on which $c'_v$ takes the same value.
  \item On full squares, the value of $\delta c'_v$ is equal to the value of $c_w$ on the
    diagonal. Indeed, a full square contains one edge from each $E(a \times b)$ and
    $E(a' \times b')$, where $a \times b$ and $a' \times b'$ are the endpoints of the diagonal
    belonging to $G^{(2)}_f$. Consequently, $\delta c'_v$ sums up the values of $c'_w$ on
    $a \times b$ and $a' \times b'$. Since $\delta c'_w = c_w$, their sum is equal to the value of
    $c_w$ on the diagonal.
  \end{enumerate}

  Thus, we have $\delta c'_v = c_v$, implying that $v_f = 0$.

  Conversely, assume that $v_f = 0$. Let $c'_v \in C^1(\Gc, \Gcf; \Z_2)$ be an
  equivariant 1-cochain such that $\delta c'_v = c_v$. We define an equivariant cochain
  $c'_w \in C^0(G^{(2)}_f; \Z_2)$ on the vertices $a \times b$ of $G^{(2)}_f$ as follows:
  \begin{enumerate}
  \item If $a$ and $b$ are regular vertices of degree one, the value of $c'_w$ on $a \times b$ is the
    sum of the values of $c'_v$ on the two edges incident to $a \times b$.
  \item Otherwise, it is the sum of the values of $c'_v$ on the two ``spine edges'', which are the edges
    forming the spine of the star of $a \times b$ (as we have discussed above, geometrically, this
    star is a book). In the case when both $a$ and $b$ are regular degree two vertices, it is
    important to carefully choose the ``spine edges'' for $a \times b$, so that our choice is
    compatible with the involution. This means that if $x \times y$ and $x' \times y'$ are the ``spine
    edges'' for $a \times b$, then the edges $y \times x$ and $y' \times x'$ must be the ``spine edges''
    for $b \times a$.
  \end{enumerate}

  As mentioned earlier, the values of $c_v$ on squares coincide with the values of $c_w$ on their
  diagonals. Thus, to establish $\delta c'_w = c_w$, it suffices to show that
  $c'_w(a \times b) + c'_w(a' \times b')$ equals the sum of the values of $c'_v$ on the edges of the
  full square containing a diagonal $\{a \times b, a' \times b'\} \in E(G^{(2)}_f)$.

  We claim that for a vertex $a \times b$ of $G^{(2)}_f$ and a full square $S$ in its star, the value of
  $c'_w$ on $a \times b$ equals the sum of the values of $c'_{v}$ on the two edges of $S$ incident
  to $a \times b$. If $a$ and $b$ are regular vertices of degree one, this follows from the definition
  of $c'_{w}$.

  Otherwise, $S$ contains exactly one of the two spine edges of $a \times b$ as a side.  Let us
  denote this edge by $e_+$, and let the other side of $S$ incident to $a \times b$ be denoted by
  $g$. Together with an empty square $S'$, $S$ forms a full page in the book of $a \times b$. The
  other spine edge $e_-$ and $g$ are the sides of $S'$ incident to $a \times b$. The other two sides
  of $S'$ do not intersect $G^{(2)}_f$, thus, $c'_v$ takes the value zero on them. Consequently,
  since $0 = c_v(S') = \delta c'_v(S') = c'_v(e_-) + c'_v(g)$, the cochain $c'_v$ takes the same
  value on the edges $g$ and $e_-$. Therefore,
  $c'_w(a \times b) = c'_v(e_+) + c'_v(e_-) = c'_v(e_+) + c'_v(g)$, where $e_+$ and $g$ are the
  sides of $S$ incident to $a \times b$.

  Therefore, for every full square $S$ containing a diagonal
  $\{a \times b, a' \times b'\} \in G^{(2)}_f$, we conclude that
  $c'_w(a \times b) + c'_w(a' \times b')$ equals the sum of the values of $c'_v$ on the sides of
  $S$. Thus, we have $\delta c'_w = c_w$ and $w_f = 0$.
\end{proof}

\begin{theorem}\label{thm:complete_2obs}
  Let $f \from T \to J$ be a stable simplicial map from a tree to a path graph. Then the following
  statements are equivalent:
  \begin{enumerate}
  \item There are no 2-obstructors for $f$;
  \item $|f|$ lifts to an embedding.
  \end{enumerate}
\end{theorem}

\begin{proof}
  If $f$ lifts to an embedding, the non-existence of 2-obstructors follows
  from~\cref{thm:obstructors}.

  Now assume that there are no 2-obstructors for $f$. Therefore, it follows
  from~\cref{lemma:equiv_obstr} that $p_2$ is trivial, and, thus, $w_f = 0$.

  Applying~\cref{thm:cohomology_equivalence}, we conclude that $v_f =
  0$. Hence,~\cref{thm:z2_approx} says that $f$ is $\Z_2$-approximable with respect to an embedding
  $|J| \hookrightarrow \R^2$. Since $T$ is a tree, by~\cref{thm:z2_imp_app}, $f$ is approximable by
  embeddings. Finally, $f$ lifts to an embedding by applying~\cref{thm:approx}.
\end{proof}

The proven result has the potential for various generalizations. For instance, we believe that the
non-existence of 2-obstructors condition remains sufficient for the existence of a lifting even for
an arbitrary, not necessarily stable, non-degenerate map from a graph to a segment.

Furthermore, instead of considering maps to a segment, we can consider maps to trees.  However, the
following example shows that the stated 2-obstructor condition is not sufficient for maps from
arbitrary graphs to trees, while the question of its sufficiency for maps from trees to trees
remains open.

\begin{figure}[h!]
  \centering
  \includegraphics[width=.99\linewidth,height=10cm,keepaspectratio]{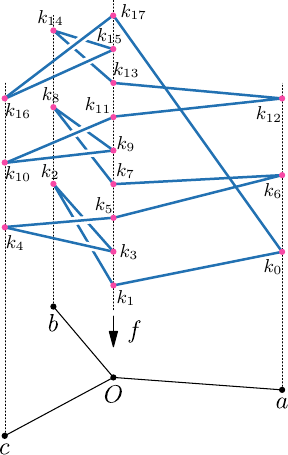}
  \caption{The 3-winding map $f \from C_{18} \to T$.}\label{fig:tripod}
\end{figure}

\begin{example}
  Let $T$ be a tripod with 4 vertices $O$, $a$, $b$, and $c$, where $O$ is the central vertex. Let
  $|f| \from S^1 \to |T|$ be a ``walking around'' $|T|$.

  Specifically, let $C_6$ be a hexagon triangulating $S^1$, whose six vertices $v_i$ numbered in
  the counterclockwise order starting from zero, and let $g \from C_6 \to T$ be simplicial map defined by
  \[
    \begin{array}{lr}
      f(v_0) = a, & f(v_1) = O, \\
      f(v_2) = b, & f(v_3) = O, \\
      f(v_4) = c, & f(v_5) = O. \\
    \end{array}
  \]

  Now, consider the composition $f \from C_{18} \xto{w} C_6 \xto{g} T$ illustrated
  in~\cref{fig:tripod}, where $C_{18}$ represents an 18-gon. Here the first map is a simplicial analogue
  of the standard 3-winding $w \from S^1 \to S^1$ defined as $z \mapsto z^3$ assuming that $S^1$ is
  embedded into the complex plane as the unit circle centered at zero. Explicitly,
  $w(k_i) = v_{\left(i \bmod 6\right)}$ where $k_0, \dots, k_{17}$ are the vertices of $C_{18}$
  numbered in the counterclockwise order.

  Note that $f$ has a 3-obstructor because $w$ does, and, therefore, it does not lift to an
  embedding.

  We claim that $f$ does not have any 2-obstructors. To demonstrate this, let us look at the graph
  $(C_{18})_f^{(2)}$.

  In this graph, a vertex corresponding to a pair $(k_i, k_j)$ of distinct vertices from $f^{-1}(O)$
  has degree two when $w(k_i) = w(k_j)$, meaning that it is also a vertex of $(C_{18})_w^{(2)}$, and
  degree one otherwise.

  Consider a vertex $(k_t, k_s)$ of $(C_{18})_f^{(2)}$ where $f(k_t) = f(k_s) \neq O$. This is a
  degree-four vertex. Moreover, two of the four vertices adjacent to it have degree one, precisely
  those that do not lie in $(C_{18})_w^{(2)}$.

  Consequently, by removing the vertices of degree one from $(C_{18})_f^{(2)}$, we get
  $(C_{18})_w^{(2)}$. Since the latter graph does not have a 2-obstructor path, the same holds true
  for the former graph.
\end{example}

\section*{Acknowledgements}

I express my strong gratitude to Sergey Melikhov for introducing me to this subject and for the
invaluable discussions and patient guidance he provided throughout my research. I am also thankful
to Olga Frolkina for her careful reading and critique of the draft version of this text, and to
Francis Lazarus and Martin Deraux for many fruitful comments on the final version.

Last but not least, I am very grateful to Petr Akhmetiev and the other participants of the Geometric
Topology Seminar at Steklov Mathematical Institute for the valuable comments and ideas expressed
during the discussion of some of the results presented in the present paper.

\bibliography{main}
\bibliographystyle{amsplain}

\end{document}